\documentclass[12pt]{amsart}
\usepackage{etex}
\usepackage{amsfonts, amssymb, latexsym}

\usepackage{amscd,amssymb}
\usepackage{verbatim}
\usepackage{pstricks}
\usepackage{pst-node}
\usepackage[mathscr]{euscript}
\usepackage{tikz}
\usetikzlibrary{matrix,arrows}
\usepackage[normalem]{ulem}
\usepackage{varwidth}
\usepackage{leftidx}

\usepackage{xcolor, fancyhdr}
\definecolor{darkmode}{RGB}{32, 31, 30}
\definecolor{gblue}{RGB}{190,209,210}
\definecolor{oblue}{RGB}{114, 160, 193}

\newsymbol\pp 1275
\usepackage{tikz}
\usetikzlibrary{matrix,arrows,backgrounds,shapes.misc,shapes.geometric,patterns,calc,positioning}

\usepackage[colorlinks=true, pdfstartview=FitV, linkcolor=blue, citecolor=blue, urlcolor=blue]{hyperref}

\usepackage{fullpage}
\usepackage{graphicx}
\usepackage{enumerate}

\def\VR{\kern-\arraycolsep\strut\vrule &\kern-\arraycolsep}
\def\vr{\kern-\arraycolsep & \kern-\arraycolsep}

\newtheorem{theorem}{Theorem}

\newtheorem{lemma}[theorem]{Lemma}
\newtheorem{prop}[theorem]{Proposition}
\newtheorem{corollary}[theorem]{Corollary}

\theoremstyle{definition}
\newtheorem{definition}[theorem]{Definition}

\newtheorem{rmk}[theorem]{Remark}
\newenvironment{remark}[1][]{\begin{rmk}[#1]\pushQED{\qed}}{\popQED \end{rmk}}

\newtheorem{qu}[theorem]{Question}

\newtheorem*{rmknonum}{Remark}

\newtheorem{obs}[theorem]{Observation}

\newtheorem{ex}[theorem]{Example}

\newcommand{\rep}{\operatorname{rep}}

\newcommand{\BL}{\mathcal{BL}}
\newcommand{\gr}{\operatorname{gr}}

\newcommand{\GL}{\operatorname{GL}}

\newcommand{\Col}{\operatorname{Col}}

\newcommand{\ZZ}{\mathbb Z}
\newcommand{\CC}{\mathbb C}
\newcommand{\RR}{\mathbb R}
\newcommand{\NN}{\mathbb N}
\newcommand{\QQ}{\mathbb Q}

\newcommand{\HH}{\mathbb H}

\newcommand{\V}{V}

\newcommand{\Id}{\mathbf{I}}

\newcommand{\rel}{\operatorname{relint}}
\newcommand{\Span}{\operatorname{Span}}
\newcommand{\Tr}{\operatorname{Tr}}

\newcommand{\ddim}{\operatorname{\mathbf{dim}}}
\newcommand{\cc}{\operatorname{\mathbf{c}}}
\newcommand{\dd}{\operatorname{\mathbf{d}}}

\newcommand{\ff}{\operatorname{\mathbf{f}}}

\newcommand{\x}{\mathbf{x}}

\newcommand{\G}{\mathcal{G}}
\newcommand{\Q}{\mathcal{Q}}

\newcommand{\ar}{\mathcal{A}}
\newcommand{\s}{\mathcal{A}}
\newcommand{\F}{\mathcal{F}}
\newcommand{\W}{\mathcal{W}}
\newcommand{\Z}{\mathcal{Z}}
\newcommand{\Eta}{\mathrm{H}}

\newcommand*\ww{\widetilde{W}}
\newcommand*\WW{\widetilde{\W}}

\newcommand{\dist}{\operatorname{dist}}
\newcommand{\capa}{\mathbf{cap}}

\newcommand{\module}{\operatorname{mod}}

\newcommand\restr[2]{{
  \left.\kern-\nulldelimiterspace 
  #1 
  \vphantom{\big|} 
  \right|_{#2} 
  }}

\begin{document}

\title{A quiver invariant theoretic approach to Radial Isotropy and the Paulsen Problem for matrix frames}
\author{Calin Chindris}
\address{University of Missouri-Columbia, Mathematics Department, Columbia, MO, USA}
\email[Calin Chindris]{chindrisc@missouri.edu}

\author{Jasim Ismaeel}
\address{University of Missouri-Columbia, Mathematics Department, Columbia, MO, USA}
\email[Jasim Ismaeel]{jmid9p@mail.missouri.edu}

\date{\today}
\bibliographystyle{amsalpha}

\begin{abstract} In this paper, we view matrix frames as representations of quivers and study them within the general framework of quiver invariant theory. We are thus led to consider the large class of semi-stable matrix frames. Within this class, we are particularly interested in radial isotropic and Parseval matrix frames. 

Using methods from quiver invariant theory \cite{ChiDer-2019}, we first prove a far reaching generalization of Barthe's Radial Isotropy Theorem \cite{Bar-1998} to matrix frames (see Theorems \ref{RIF-thm-1}{(3)} and \ref{quiver-RadIso-thm}). With this tool at our disposal, we provide a quiver invariant theoretic approach to the Paulsen problem for matrix frames. We show in Theorem \ref{Paulsen-MF-thm} that for any given $\varepsilon$-nearly equal-norm Parseval frame $\F$ of $n$ matrices with $d$ rows there exists an equal-norm Parseval frame $\W$ of $n$ matrices with $d$ rows such that $\dist^2(\F,\W)\leq 46 \varepsilon d^2$. 
\end{abstract}

\maketitle
\setcounter{tocdepth}{1}

\section{Introduction} 

\subsection{Motivation} Our motivation in this paper goes back to the Paulsen problem which asks to determine how close  an $\varepsilon$-nearly equal-norm Parseval frame is to an equal-norm Parseval frame. The significance of this problem stems from the fact that, on the one hand, there exist algorithms for constructing $\varepsilon$-nearly equal-norm Parseval frames and on the other hand, there is a wide range of applications of equal-norm Parseval frames. Given the high level of difficulty in constructing equal-norm Parseval frames, one would like to know how far a given $\varepsilon$-nearly equal-norm Parseval frame is from the set of all equal-norm Parseval frames.

A systematic study of the Paulsen problem has been undertaken by Casazza and his collaborators in \cite{BodCas-2010, CasFicMix-2012, Cas-2013, CahCas-2013}. The first upper bound on the squared distance between an $\varepsilon$-nearly equal-norm Parseval frame and the set of all equal-norm Parseval frames that does not depend on the number of vectors of the frames involved was found by Kwok, Lau, Lee, and Ramachandran in \cite{KwoLauLeeRam2018}. Later on, Hamilton and Moitra \cite{HamMoi-2018} have found a significantly shorter solution to the Paulsen problem. One of the key ingredients in Hamilton-Moitra's approach is a theorem of Barthe \cite{Bar-1998} on vectors in radial isotropic position (see also \cite{For-2002}). The use of such vectors can be traced back to Fritz John's seminal work \cite{John-1948} which led to the notion of the John ellipsoid of a convex body. Since then radial isotropy for vectors has found applications to numerous other areas such as the complexity of unbounded error probabilistic communication protocols \cite{For-2002}; algorithmic and optimization aspects of Brascamp-Lieb inequalities \cite{BenCarChrTao-2008}; superquadratic lower bounds for $3$-query correctable codes \cite{DviSarWig-2017}; and algorithmic aspects of point location in high-dimensional arrangements of hyperplanes \cite{KanLovMor-2018}. For an excellent account on vectors in radial isotropic position, we refer the reader to the recent work of Artstein-Avidan, Kaplan, and Sharir \cite{Art-AviKapHai-2020}. 

Going beyond the classical set-up of frames, fusion frames and, more generally, matrix frames are the main tools for applications to distributed sensing \cite{KutPezCalLiu-2009, PezKutCal-2008}, parallel processing \cite{BjoMan-1991}, and packet encoding \cite{Bod-2007}, just to name a few. It is therefore an important task to establish an analog of the Paulsen problem for \emph{matrix frames} and design effective ways for the construction of $\varepsilon$-nearly equal-norm Parseval matrix frames. In fact, in this paper, we develop a quiver invariant theoretic approach to radial isotropy and applications to the Paulsen problem for such frames.

Let $d,n, d_1, \ldots, d_n$ be fixed positive integers with $n>d$, and let $\F=\{X_1, \ldots, X_n\}$ be a collection of matrices with $X_i \in \RR^{d \times d_i}$ for all $i \in [n]$. We encode $\F$ as a representation of the bipartite quiver $\Q$ as follows
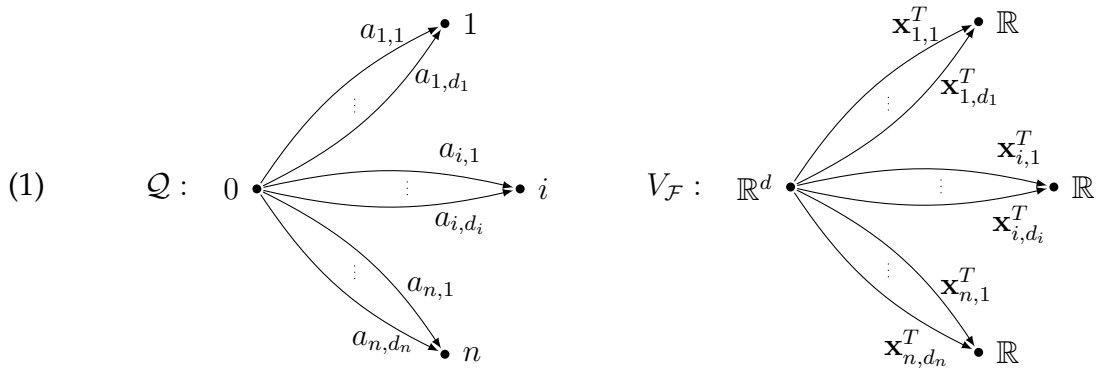
\begin{equation}\label{quiver-eqn-1}
\Q:~
\vcenter{\hbox{  
\begin{tikzpicture}[point/.style={shape=circle, fill=black, scale=.3pt,outer sep=3pt},>=latex]
   \node[point,label={left:$0$}] (1) at (-2.5,0) {};
   \node[point,label={right:$1$}] (2) at (0,2.2) {};
   \node[point,label={right:$i$}] (3) at (1,0) {};
   \node[point,label={right:$n$}] (4) at (0,-2.2) {};
  
   \draw[dotted] (-0.5,.1)--(-0.5,-.1);
   \draw[dotted] (-1.2,1.2)--(-1.2,1);
   \draw[dotted] (-1.2,-1.2)--(-1.2,-1);

   \path[->]
   (1) edge [bend left=15] node[pos=.75, above] {$a_{1,1}$} (2)
   (1) edge [bend right=15] node[pos=.75, right] {$a_{1, d_1}$} (2)
   (1) edge [bend left=13] node[pos=.8, above] {$a_{i,1}$} (3)   
   (1) edge [bend right=13] node[pos=.8, below] {$a_{i,d_i}$} (3) 
   
   (1) edge [bend left=15]  node[pos=.7, right] {$a_{n,1}$} (4)
   (1) edge [bend right=15] node[pos=.75, below] {$a_{n,d_n}$} (4);
\end{tikzpicture} 
}}
\hspace{30pt}
\V_{\F}:~
\vcenter{\hbox{  
\begin{tikzpicture}[point/.style={shape=circle, fill=black, scale=.3pt,outer sep=3pt},>=latex]
   \node[point,label={left:$\RR^d$}] (1) at (-2.5,0) {};
   \node[point,label={right:$\RR$}] (2) at (0,2.2) {};
   \node[point,label={right:$\RR$}] (3) at (1,0) {};
   \node[point,label={right:$\RR$}] (4) at (0,-2.2) {};
  
   \draw[dotted] (-0.5,.1)--(-0.5,-.1);
   \draw[dotted] (-1.2,1.2)--(-1.2,1);
   \draw[dotted] (-1.2,-1.2)--(-1.2,-1);

   \path[->]
   (1) edge [bend left=15] node[pos=.75, above] {$\x^T_{1,1}$} (2)
   (1) edge [bend right=15] node[pos=.70, right] {$\x^T_{1, d_1}$} (2)
   (1) edge [bend left=13] node[pos=.9, above] {$\x^T_{i,1}$} (3)   
   (1) edge [bend right=13] node[pos=.9, below] {$\x^T_{i, d_i}$} (3) 
   
   (1) edge [bend left=15]  node[pos=.7, right] {$\x^T_{n,1}$} (4)
   (1) edge [bend right=15] node[pos=.75, below] {$\x^T_{n, d_n}$} (4);
\end{tikzpicture} 
}}
\end{equation}
where $\x_{i,\ell} \in \RR^d$ is the $\ell$th column of the matrix $X_i$ for all $i \in [n]$ and $\ell\in [d_i]$. The quiver $\Q$ is the bipartite directed graph with one source vertex, labeled by $0$, and $n$ sink vertices labeled by $1, \ldots, n$; moreover, there are precisely $d_i$ arrows from $0$ to $i$ for each $i \in [n]$. Guided by quiver invariant theoretic considerations, we associate to $\F$ its \emph{orbit polytope} defined by
\begin{equation*}
K_{\F}:=\left\{ \cc=(c_1, \ldots, c_n) \in \RR^n_{\geq 0}  \;\middle|\;
\begin{array}{l}
c_1+\ldots +c_n=d, \text{~and~}\\
\sum_{i \in I} c_i \leq \dim (\sum_{i \in I} \Col(X_i)), \forall I \subseteq [n]
\end{array} \right\},
\end{equation*}
where $\Col(X_i)$ denotes the subspace of $\RR^d$ spanned by the columns of $X_i$ for each $i \in [n]$. (When $d_1=\ldots=d_n=1$, $K_{\F}$ is the basis polytope associated to the set of vectors $\F$.)

We say that $\F$ is a \emph{matrix frame}\footnote{A matrix frame without any other structure is precisely a frame. Indeed, the columns of the matrices of an MF form a frame and, conversely, any frame can be converted into an MF by partitioning the set of frame vectors into subsets of appropriate size. It is only after adding additional structure to an MF (e.g. semi-stable MFs, PMFs, RIFs) that the two notions begin to diverge.} (``MF") if
\begin{equation*}\label{eqn-mf}
\sum_{i=1}^n X_i  X_i^T \text{~is a positive definite matrix}.
\end{equation*}

Now let $\cc=(c_1, \ldots, c_n) \in \QQ^n_{> 0}$ be an $n$-tuple of positive rational\footnote{The assumption that the weights are positive rational numbers allows us to define the integral stability weight $\sigma_{\cc}$ (see equation (\ref{eqn-induced-wt})) which we need in order to run the quiver invariant theory machinery.} weights. We call $(\F, \cc)$ a \emph{weighted} MF if the the weighted sum $\sum_{i=1}^n c_i X_i  X_i^T$ is a positive definite matrix. This can be easily seen to be equivalent to $\F$ being an MF since the weights $c_i$ are positive.

The datum $(\F, \cc)$ is said to be a \emph{semi-stable matrix frame} if $\cc \in K_\F$. We point out that there exists a deterministic polynomial time algorithm for checking whether $(\F, \cc)$ is a semi-stable matrix frame (see Remark \ref{semi-stab-MF-polytime} for more details).

We say that $(\F, \cc)$ is a weighted \emph{Parseval matrix frame} (``PMF") if
\begin{equation*}\label{eqn-pmf}
\sum_{i=1}^n c_iX_i  X_i^T=I_d \text{~and~} \lVert X_i\rVert_F^2=1, \forall i\in[n].
\end{equation*}
Working with the columns of the matrices involved, any weighted PMF can be viewed as a weighted Parseval frame with the property that the vectors of the frame can be partitioned into $n$ groups so that the sum of the squared norms of the vectors within each group is equal to one. A closely related notion is that of a radial isotropic frame. We call $(\F, \cc)$ a \emph{radial isotropic matrix frame} (``RIF") if
\begin{equation} \label{eqn-rif}
\sum_{i=1}^n c_i\frac{X_i  X_i^T}{\lVert X_i\rVert^2_F}=I_d.
\end{equation}
We say that $(\F, \cc)$ can be \emph{transformed into a RIF} if there exists an invertible matrix $A \in \GL(d)$ such that $(\{A X_i \}_{i \in [n]}, \cc)$ is a RIF.  When $d_1=\ldots=d_n=1$, $(\ref{eqn-rif})$ recovers the concept of vectors in (or that can brought into) radial isotropic position. In Remark \ref{rmk-classical-vs-matrix-RadIso}, we point out some of the difficulties that one encounters when trying to use classical radial isotropy (i.e. the case when $d_1=\ldots=d_n=1$) to transform a matrix frame into a RIF. This remark also points to the need for new techniques, like the quiver invariant theoretic techniques developed in this paper, in order to deal with the general case of matrix frames.

We have the following basic dictionary translating between key concepts in frame theory and quiver invariant theory:
\begin{table}[h!]
  \begin{center}
    \label{tab:table1}
    \begin{tabular}{|l||l|} 
    \hline
      \emph{\textbf{Frame theory}} & \emph{\textbf{Quiver invariant theory}}\\
      \hline
      \hline
      Frames & Quiver representations\\
      \hline
      Orbit polytopes of matrix frames & Orbit cones of quiver representations\\
      \hline
      Semi-stable matrix frames &  Semi-stable quiver representations\\
      \hline
      Parseval matrix frames  &  Geometric Brascamp-Lieb quiver data\\
      \hline
      Frames that can be transformed into a RIF & Polystable quiver representations\\
      \hline
    \end{tabular}
  \end{center}
\end{table}

\noindent
Moreover, any PMF is a  semi-stable MF, and any semi-stable MF is a matrix frame (see Corollary \ref{pmf-semi-stab-MF-coro} for details).

\subsection{Our results} Let $\F=\{X_1, \ldots, X_n\}$ be a matrix frame with $X_i \in \RR^{d \times d_i}$, $i \in [n]$. Let $\Phi_\F:\RR^n\longrightarrow \RR$ be the function defined  by
$$ 
\Phi_\F(t)=\log\left(\det\left(\sum_{i=1}^n e^{t_i}X_iX_i^T\right)\right) \text{ where }t=(t_i)_{i=1}^n \in \RR^n,
$$ 
and let us consider $f_\F:\RR^n \longrightarrow\RR\cup\{-\infty\}$ defined by 
$$
f_\F(\cc)=\inf_{t\in\RR^n}\{\Phi_\F(t)-\langle t,\cc\rangle\},$$ 
i.e. $f_\F=-\Phi_\F^*$ where $\Phi_\F^*$ is the Legendre dual of $\Phi_\F$. The function $f_\F$ can also be obtained from the capacity of the quiver datum $(V_\F, \cc)$ (see Remark \ref{cap-matrix-frame-def} for details).  In what follows, $\rel(K)$ denotes the relative interior of a polytope $K$.

We are now ready to state our first result.

\begin{theorem} \label{RIF-thm-1} Let $\F=\{X_1, \ldots, X_n\}$ be a matrix frame with $X_i \in \RR^{d \times d_i}$, $i \in [n]$, and let $\cc=(c_1, \ldots, c_n) \in \QQ_{>0}^n$ be positive rational weights such that  $\sum_{i=1}^n c_i=d$. Then the following statements hold.
\begin{enumerate}
\smallskip
\item The domain of $\restr{f_\F}{\QQ_{>0}^n}$ is precisely $K_\F \cap \QQ_{>0}^n$.

\smallskip
\item (\textbf{The RIF-degeneration of a semi-stable frame}) Assume that $(\F, \cc)$ is a semi-stable frame. Then $(\F, \cc)$ degenerates (in the sense of quiver representations) to a semi-stable frame $(\widetilde{\F}, \cc)$ such that $(\widetilde{\F}, \cc)$ can be transformed into a RIF.

\smallskip
\item (\textbf{Matrix Radial Isotropy}) Assume that $\V_\F$ is a locally semi-simple representation. Then the following statements are equivalent:
\begin{enumerate}[(i)]
\item $\cc \in \rel(K_\F)$;
\item $(\F, \cc)$ can be transformed into a RIF;
\item $f_\F(\cc)$ is finite and is attained at some point of $\RR^n$. 
\end{enumerate}
\end{enumerate}
\end{theorem}
\noindent
We point out that a sufficient condition for $\V_\F$ to be a locally semi-simple representation is that any $d$ of the $N$ columns of the matrix $[X_1 | \ldots | X_n] \in \RR^{d \times N}$ form a basis for $\RR^d$ where $N=d_1+\ldots+d_n$ (see Lemma \ref{ex-gen-locally-semi-simple} for details). Such matrices are also known as full spark in the compressed sensing literature (see for example \cite{AleCahMix-2012}). For a more general version of Theorem \ref{RIF-thm-1}{(3)}, see the Quiver Radial Isotropy Theorem \ref{quiver-RadIso-thm}.

\bigskip
We now come to the Paulsen problem for matrix frames. We say that $\F=\{X_1, \ldots, X_n\}$ with $X_i \in \RR^{d \times d_i}$, $i \in [n]$, is an \emph{equal-norm PMF} if $\left(\left\{\sqrt{\frac{n}{d}}X_i\right\}_{i\in[n]}, \cc=\left(\frac{d}{n}\right)_{i \in [n]} \right)$ is a weighted PMF,  i.e.
$$\sum_{i=1}^n X_i  X_i^T=I_d \text{~and~} \lVert X_i\rVert^2_F=\frac{d}{n}, \forall i\in[n].$$

\smallskip
\noindent
We point out that when $d_1=\ldots=d_n$, Lemma 12 in \cite{Bod-2007} states that among all weighted matrix frames $(\F, \cc)$ with $\left(\left\{c_i^{-{1 \over 2}}X_i\right\}_{i\in[n]}, \cc \right)$ a weighted PMF, the ones most robust to a single erasure, i.e., losing $X_i^Tu$ of the signal $u$ we are analyzing for some $i \in [n]$, are the equal-norm PMFs where the columns of each $X_i$ are of equal norm and orthogonal to each other. This result when $d_1=\ldots=d_n=1$ served as one of the motivations for the original Paulsen problem. 

Given $\varepsilon>0$, a matrix frame $\F$ is said to be an \emph{$\varepsilon$-nearly equal-norm PMF} if 
\[(1-\varepsilon)I_d\preceq\sum_{i=1}^n X_i  X_i^T\preceq(1+\varepsilon)I_d\text{\quad and \quad}(1-\varepsilon)\frac{d}{n}\leq \lVert X_i\rVert^2_F \leq(1+\varepsilon)\frac{d}{n}, \forall i\in[n].\]

We are interested to know how close  an $\varepsilon$-nearly equal-norm PMF is to an equal-norm PMF. To this end, given two $n$-tuples of matrices $\F=\{X_1,\ldots,X_n\}$ and $\G=\{Y_1,\ldots,Y_n\}$ with $X_i,Y_i\in \RR^{d\times d_i}, i \in [n]$, define
$$
\dist^2(\F,\G)=\sum_{i=1}^n \lVert X_i-Y_i\rVert_F^2.
$$

Using the Matrix Radial Isotropy Theorem \ref{RIF-thm-1}{(3)} in an essential way, we can extend Hamilton-Moitra's upper bound \cite{HamMoi-2018} from the case of frames (of vectors) to that of frames of matrices. 

\begin{theorem}[\textbf{The Paulsen problem for matrix frames}] \label{Paulsen-MF-thm} Let $\F=\{X_1,\ldots,X_n\}$ with $X_i \in \RR^{d\times d_i}$, $i \in [n]$, be an $\varepsilon$-nearly equal-norm PMF with $\varepsilon<0.3$. Then there exists an equal-norm PMF $\W=\{W_1,\ldots,W_n\}$ with $W_i \in \RR^{d\times d_i}$, $i \in [n]$, such that 
$$
\dist^2(\F,\W)\leq 46 \varepsilon d^2.
$$ 
\end{theorem} 
\noindent

Our next result addresses the constructive aspects of Theorem \ref{RIF-thm-1}{(3)}, extending those  in \cite{Bar-1998, Art-AviKapHai-2020} from frames to matrix frames. To state it, we need to introduce the following objects. Let
\begingroup\makeatletter\def\f@size{11}\check@mathfonts
$$
\s=\left \{ (I, (S_i)_{i \in I}) \;\middle|\;
\begin{array}{l}
\emptyset \neq I \subseteq [n],\\
\emptyset \neq S_i \subseteq [d_i], \forall i \in I, \text{~and~}\\
\sum_{i \in I} |S_i|=d
\end{array}    \right \}.
$$
\endgroup
(When $d_1=\ldots=d_n=1$, $\s$ is just the set of all subsets $I \subseteq [n]$ of size $d$.) For $S=(\ell, (S_{\ell})_{{\ell} \in I}) \in \s$ and $i \in [n]$, we write $i \in S$ to mean that $i \in I$. Furthermore, we define
$$
\Delta_S:=\det\left( \sum_{\ell \in S} X[S_{\ell}]\cdot X[S_{\ell}]^T \right),
$$
where $X[S_{\ell}]$ consists of the columns of the matrix $X_{\ell}$ indexed by the elements of the subset $S_{\ell} \subseteq [d_{\ell}]$ for each $\ell \in S$. We also need the following semi-algebraic set
\begingroup\makeatletter\def\f@size{10.65}\check@mathfonts
$$
\mathbb{V}(\F, \cc):=\left \{(\xi_1, \ldots, \xi_n) \in \RR_{>0}^n \;\middle|\; \sum_{S \in \s, i \in S} |S_i| \left( \prod_{\ell \in S} \xi_{\ell}^{|S_{\ell}|} \right) \Delta_S=c_i \left( \sum_{S \in \s} \left( \prod_{\ell \in S} \xi_{\ell}^{|S_{\ell}|} \right)  \Delta_S \right), \forall i \in [n]  \right \}.
$$
\endgroup

Finally, we can state our next result which opens up the possibility of using gradient descent and methods from real algebraic geometry to find a matrix that transforms $(\F, \cc)$ into a RIF.  

\begin{theorem}[\textbf{Constructive aspects of RIFs}] \label{RIF-thm-2} Let $(\F, \cc)$ be a matrix frame where $\F=\{X_1, \ldots, X_n\}$ with $X_i \in \RR^{d \times d_i}$, $i \in [n]$, and $\cc=(c_1, \ldots, c_n) \in \QQ_{>0}^n$, with $\sum_{i=1}^n c_i=d$. Let $t^*=(t_1^*, \ldots, t_n^*) \in \RR^n$ and set
$$
Q(t^*):=\sum_{i=1}^n e^{t^*_i} X_i X_i^T.
$$
Then the following statements are equivalent:
\begin{enumerate}
\item $f_\F(\cc)$ is finite and attained at $t^*$;
\item $(e^{t^*_1}, \ldots, e^{t^*_n}) \in \mathbb{V}(\F, \cc)$.
\end{enumerate}
If either $(1)$ or $(2)$ holds then $Q^{-{1 \over 2}}(t^*)$ transforms $(\F, \cc)$ into a RIF.
\end{theorem}

\section{Background on Quiver Invariant Theory}\label{QIT-RIF-sec}
Throughout, we work over the field $\RR$ of real numbers and denote by $\NN=\{0,1,\dots \}$. For a positive integer $L$, we denote by $[L]=\{1, \ldots, L\}$.

A \emph{quiver} $Q=(Q_0,Q_1,t,h)$ consists of two finite sets $Q_0$ (\emph{vertices}) and $Q_1$ (\emph{arrows}) together with two maps $t:Q_1 \to Q_0$ (\emph{tail}) and $h:Q_1 \to Q_0$ (\emph{head}). We represent $Q$ as a directed graph with set of vertices $Q_0$ and directed edges $a:ta \to ha$ for every $a \in Q_1$. 

A \emph{representation} of $Q$ is a family $V=(V(x), V(a))_{x \in Q_0, a\in Q_1}$ where $V(x)$ is a finite-dimensional $\RR$-vector space for every $x \in Q_0$, and $V(a): V(ta) \to V(ha)$ is an $\RR$-linear map for every $a \in Q_1$. After fixing bases for the vector spaces $\V(x)$, $x \in Q_0$, we often think of the linear maps $\V(a)$, $a \in Q_1$, as matrices of appropriate size. A \emph{subrepresentation} $W$ of $V$, written as $W \subseteq V$, is a representation of $Q$ such that $W(x) \subseteq V(x)$ for every $x \in Q_0$, and moreover $V(a)(W(ta)) \subseteq W(ha)$ and $W(a)=\restr{V(a)}{W(ta)}$ for every arrow $a \in Q_1$. The \emph{direct sum} of two (or more) quiver representations $V_1$ and $V_2$, denoted by $V_1 \oplus V_2$, is the quiver representation defined by $(V_1 \oplus V_2)(x)=V_1(x) \oplus V_2(x), \forall x \in Q_0$, and  $(V_1 \oplus V_2)(a)=V_1(a) \oplus V_2(a), \forall a \in Q_1$.

The dimension vector $\ddim V \in \NN^{Q_0}$ of a representation $V$  is defined by $\ddim V(x)=\dim_{\RR} V(x)$ for all $x \in Q_0$. By a dimension vector of $Q$, we simply mean a $\NN$-valued function on the set of vertices $Q_0$. For two vectors $\sigma, \dd \in \RR^{Q_0}$, we denote by $\sigma \cdot \dd$ the scalar product of the two vectors, i.e. $\sigma \cdot \dd=\sum_{x \in Q_0} \sigma(x)\dd(x)$. 

Let $\sigma \in \RR^{Q_0}$ be a real weight of $Q$. A representation $\V$ of $Q$ is said to be \emph{$\sigma$-semi-stable} if 
\begin{equation}\label{semi-stab-eqn-1}
\sigma \cdot \ddim \V=0 \text{~and~} \sigma \cdot \ddim \V'\leq 0, ~\forall \V' \subseteq \V.
\end{equation}
We say that $\V$ is \emph{$\sigma$-stable} if $\sigma \cdot \ddim \V=0 $ and the inequalities in $(\ref{semi-stab-eqn-1})$ are strict for all proper subrepresentations  $0 \neq V' \subsetneq V$. We point out that, as shown by King \cite{K}, the linear homogeneous inequalities $(\ref{semi-stab-eqn-1})$ come from the Hilbert-Mumford's numerical criterion for semi-stability applied to quiver representations. We call a representation \emph{$\sigma$-polystable} if it is a finite direct sum of $\sigma$-stable representations. By a \emph{locally semi-simple} representation we mean a representation $V$ of $Q$ such that $V$ is $\sigma_0$-polystable for some weight $\sigma_0 \in \ZZ^{Q_0}$.

\begin{lemma}[\textbf{Generic tuples as locally semi-simple representations}] \label{ex-gen-locally-semi-simple} Let $\F=\{X_1, \ldots, X_n\}$ be an $n$-tuple of matrices  with $X_i \in \RR^{d \times d_i}$ for all $i \in [n]$, and such that $n>d$. Assume that $\F$ is generic, meaning that any $d$ of the columns of the matrix $[X_1| \ldots |X_n]$ form a basis of $\RR^d$. 

Let $\Q$ be the quiver from $(\ref{quiver-eqn-1})$ and view $\F$ as the representation $\V_\F$ of $\Q$. Then $V_\F$ is $\sigma_0$-stable, and hence locally semi-simple, where $\sigma_0$ is the weight of $\Q$ defined by
$$
\sigma_0(0)=n, \text{~and~} \sigma_0(i)=-d, \forall i \in [n].
$$
\end{lemma}

\begin{proof}
For any subrepresentation $W$ of $V_\F$, let us denote $W(0)$ by $T$ and define
$$
I=\{i\in[n]\mid \Col(X_i) \subseteq T^{\perp}\}.
$$
Next, we associate to the subspace $T\subseteq \RR^d$ the subrepresentation $W_T \subseteq V_\F$ defined by

\[W_T(0)=T\text{ and } W_T(i)=\begin{cases}
\RR &\text{ if }i\not\in I\\
0 &\text{ if }i\in I 
\end{cases},\]
for all $i\in[n]$. Then we have that
\begin{align*}
\sigma_0 \cdot \ddim W \leq \sigma_0\cdot\ddim W_T&=n\dim(T)-d(n-|I|)\\
&=n(d-\dim (T^\perp))-dn+d|I|.
\end{align*}

Now let us assume that $W$ is a non-zero, proper subprepresentation of $V_\F$. We will check that $\sigma_0\cdot \ddim W<0$. First, it is clear that if $I=\emptyset$ then $W=W_T$ and 
$$\sigma_0 \cdot \ddim W=n(\dim(T)-d)<0,$$ 
since $T$ must be a proper subspace of $\RR^d$ as otherwise $W=V_\F$, contradicting our assumption that $W \neq V_\F$. If $I$ is non-empty, we distinguish the following two cases.

\smallskip
\noindent
\emph{Case 1}: If $|I|\geq d$ then $T^{\perp}$ contains $d$ linearly independent vectors of the form $\x_{i,\ell}$ since $\F$ is generic, and so $T=\{0\}$.  Therefore, 
$$
\sigma_0\cdot \ddim W=-d\left( \sum_{i \in [n]} \dim W(i)\right)<0,
$$
since $W$ is not the zero representation.

\smallskip
\noindent
\emph{Case 2}: If $0<|I|<d$ then $|I| \leq \dim(T^{\perp})$ and thus
$$
\sigma_0\cdot \ddim W\leq n(d-|I|)-nd+d|I|=|I|(d-n)<0.
$$
\end{proof}

From now on,  we assume that $Q$ is a bipartite quiver. This means that $Q_0$ is the disjoint union of two subsets $Q^+_0=\{\overline{1}, \ldots, \overline{m}\}$ and $Q^-_0=\{1, \ldots, n\}$, and all arrows in $Q$ go from $Q^+_0$ to $Q^-_0$. For all $j \in [m]$ and $i \in [n]$, we denote by $\ar_{j,i}$ the set of all arrows from vertex $\overline{j}$ to vertex $i$.

Let $\cc=(c_1, \ldots, c_n) \in \QQ^n_{>0}$ an $n$-tuple of positive rational numbers and 
$\omega$ the least common denominator of $c_1, \ldots, c_n$. Then the weight $\sigma_{\cc} \in \ZZ^{Q_0}$ of $Q$ induced by $\cc$ is defined by 
\begin{equation} \label{eqn-induced-wt}
\sigma_{\cc}(\overline{j})=\omega, \forall j \in [m], \text{~and~}\sigma_{\cc}(i)=-\omega c_i, \forall i \in [n].
\end{equation}
In what follows, we say that a representation $\V$ of $Q$ is \emph{$\cc$-(semi-)stable} if and only if $V$ is $\sigma_{\cc}$-(semi-)stable. We say that $V$ is \emph{$\cc$-polystable} if and only if $V$ is $\sigma_{\cc}$-polystable.

For a dimension vector $\dd \in \NN^{Q_0}$,  the representation space of $\dd$-dimensional representations of $Q$ is the vector space
$$
\rep(Q, \dd):=\prod_{a \in Q_1} \RR^{\dd(ha)\times \dd(ta)}.
$$ 
The group $\GL(\dd):=\prod_{x \in Q_0} \GL(\dd(x))$ acts on $\rep(Q, \dd)$ by simultaneous conjugation, i.e.  for $A=(A(x))_{x \in Q_0} \in \GL(\dd)$ and $V=(V(a))_{a \in Q_1} \in \rep(Q,\dd)$,
$$
(A \cdot V)(a):=A(ha) V(a)  A(ta)^{-1},  \forall a \in Q_1. 
$$
Then the $\GL(\dd)$-orbits in $\rep(Q,  \dd)$ are in a one-to-one correspondence with the isomorphism classes of $\dd$-dimensional representations of $Q$.

\begin{definition}[\textbf{Geometric Brascamp-Lieb quiver data}] (see \cite{ChiDer-2019}) Let $\dd \in \NN^{Q_0}$ be a dimension vector, $V \in \rep(Q,\dd)$ a representation of $Q$ of dimension vector $\dd$, and $\cc=(c_1, \ldots, c_n) \in \QQ^n_{>0}$ an $n$-tuple of positive rational numbers. We call $(V, \cc)$ a \emph{geometric BL quiver datum} if 
\begin{equation*} \label{BL-geom-eq-3}
\sum_{i=1}^n c_i \sum_{a \in \ar_{j,i}} V(a)^T \cdot V(a)=\Id_{\dd(\overline{j})}, \forall j \in [m],
\end{equation*}
and
\begin{equation*} \label{BL-geom-eq-4}
\sum_{j=1}^m  \sum_{a \in \ar_{j,i}} V(a) \cdot V(a)^T=\Id_{\dd(i)}, \forall i \in [n].
\end{equation*} 
We say that $(V, \cc)$ can be transformed into a geometric BL quiver datum if there exists $A \in \GL(\dd)$ such that $(A \cdot V, \cc)$ is a geometric BL quiver datum.
\end{definition}

\begin{rmk}
We point out that when $Q$ is the $n$-subspace quiver, this is the definition of a geometric BL datum introduced by Bennett, Carbery, Christ, and Tao in \cite[Section 2]{BenCarChrTao-2008}.
\end{rmk}

\begin{rmk}[\textbf{Geometric BL quiver data and Parseval matrix frames}] \label{Geo-Bl-data-Parseval} Let $\F=\{X_1, \ldots, X_n\}$ be an $n$-tuple of matrices  with $X_i \in \RR^{d \times d_i}$, $i \in [n]$.  Let $\Q$ be the quiver from $(\ref{quiver-eqn-1})$ and view $\F$ as the representation $\V_\F$ of $\Q$.  Then, for an $n$-tuple of positive rational numbers $\cc \in \QQ^n_{>0}$, it is immediate to see that
$(\F, \cc)$ is a PMF if and only if $(\V_\F,  \cc)$ is a geometric BL quiver datum.
\end{rmk}

\begin{definition}[\textbf{The orbit polytope of a representation}] Let $Q$ be a bipartite quiver with set of source vertices $Q^+_0=\{\overline{1}, \ldots, \overline{m}\}$ and of sink vertices $Q_0^-=\{1, \ldots, n\}$.  Let $\dd$ be a dimension vector of $Q$ and $V \in \rep(Q, \dd)$ a $\dd$-dimensional representation of $Q$. We define the \emph{orbit polytope} of $V$ to be the polytope
\begin{equation*}
K_V:=\left \{ \cc=(c_1, \ldots, c_n) \in \RR^n_{\geq 0} \;\middle|\; 
\begin{array}{l} \sum_{j=1}^m \dd(\overline{j})=\sum_{i=1}^n c_i \dd(i), \text{~and}\\
\\
\sum_{j=1}^m \dim W(\overline{j}) \leq \sum_{i=1}^n c_i \dim W(i), \\
\\
\text{for all subrepresentations~} W \subseteq V
\end{array}
\right \}.
\end{equation*}
\end{definition}

\begin{remark}\label{polytope-vs-orbit-cone-rmk} Note that for an $n$-tuple $\cc=(c_1, \ldots, c_n) \in \QQ^n_{> 0}$ with associated weight $\sigma_{\cc}$, $V$ is $\sigma_{\cc}$-semi-stable if and only if $\cc \in K_V$.  Thus $K_V$ is a cross-section of the orbit cone of $V$ (see Definition \ref{orbit-cone-defn}) by certain hyperplanes.
\end{remark}

\begin{lemma}[\textbf{The orbit polytope of a tuple of matrices}] \label{orbit-poly-tuples-lemma} Let $\F=\{X_1, \ldots, X_n\}$ be an $n$-tuple of matrices  with $X_i \in \RR^{d \times d_i}$ for all $i \in [n]$.  Let $\Q$ be the quiver from $(\ref{quiver-eqn-1})$ and view $\F$ as the representation $\V_\F$ of $\Q$.  Then
$$
K_{V_\F}=\left \{ \cc=(c_1, \ldots, c_n) \in \RR^n_{\geq 0}  \;\middle|\;
\begin{array}{l}
c_1+\ldots +c_n=d, \text{~and~}\\
\sum_{i \in I} c_i \leq \dim (\sum_{i \in I} \Col(X_i)), \forall I \subseteq [n]
\end{array}    \right \},
$$
i.e.  $K_{V_\F}=K_\F$.
\end{lemma}

\begin{proof} Since the dimension vector of $V_\F$ is equal to one at the sink vertices of $\Q$, one can easily describe the subrepresentations of $V_\F$ that really matter when checking whether a vector $\cc$ belongs to $K_{V_\F}$. In what follows, for a representation $W$ of $\Q$ and $\cc \in \RR^n$, we set
$$
\cc(W):=\dim W(0)-\sum_{i=1}^n c_i \dim W(i).
$$
Then $K_{V_\F}=\left \{ \cc=(c_1, \ldots, c_n) \in \RR^n_{\geq 0} \;\middle|\; 
c_1+\ldots +c_n=d \text{~and~} \cc(W) \leq 0, \forall W \subseteq V_\F
\right \}$. Now, for every subset $I \subseteq [n]$, consider the subrepresentation $W_I$ of $V_\F$ defined by 
$$
W_I(0)=\left(\sum_{i\in I}\Col(X_i)\right)^\perp \text{ and } W_I(i)=\begin{cases}
\RR &\text{ if }i\not\in I\\
0 &\text{ if }i\in I
\end{cases}
$$
for all $i \in [n]$. Then, for any $\cc=(c_1, \ldots, c_n) \in \RR^n_{\geq 0}$ with $c_1+\ldots +c_n=d$, it is immediate to check that 
\begin{equation}\label{eqn-1-orb-poly-tuples}
\cc(W_I)=\sum_{i \in I} c_i-\dim \left(\sum_{i \in I} \Col(X_i)\right).
\end{equation}

Conversely, for any subrepresentation $W \subseteq V_\F$, define 
$$
I_W:=\{i \in [n] \mid \Col(X_i) \subseteq W(0)^{\perp}\} \subseteq [n].
$$
Then one can easily check that
\begin{equation}\label{eqn-2-orb-poly-tuples}
\cc(W) \leq \cc(W_{I_W}).
\end{equation}

It now follows from ($\ref{eqn-2-orb-poly-tuples}$) that 
$$
K_{V_\F}=\left \{ \cc=(c_1, \ldots, c_n) \in \RR^n_{\geq 0} \;\middle|\; 
c_1+\ldots +c_n=d \text{~and~} \cc(W_I) \leq 0, \forall I \subseteq [n]
\right \},
$$ 
which together with ($\ref{eqn-1-orb-poly-tuples}$) finishes the proof.
\end{proof}

We now come to the concept of the capacity of quiver data that is essential for our study of matrix frames.  It has been introduced in \cite{ChiDer-2019} as a way to build a bridge between quiver invariant theory and Brascamp-Lieb theory in harmonic analysis.

\begin{definition}[\textbf{The capacity of a quiver datum}](compare to \cite[Lemma 8]{ChiDer-2019}) Let $(V, \cc)$ be a quiver datum with $V \in \rep(Q, \dd)$ and $\cc \in \RR^n_{>0}$. 

\begin{enumerate}[(i)]
\item We define the \emph{capacity} of the quiver datum $(V, \cc)$ to be the non-negative real number
\begingroup\makeatletter\def\f@size{11}\check@mathfonts
\begin{equation} \label{cap-eqn}
\capa(V, \cc):=  \inf \left \{ {\prod_{j=1}^m \det\left( \sum_{i=1}^n c_i \left(\sum_{a \in \ar_{j,i}}V(a)^T \cdot Y_i \cdot V(a) \right) \right) \over \prod_{i=1}^n \det(Y_i)^{c_i}} \;\middle|\; Y_i \in \mathcal{S}^{+}_{\dd(i)} \right \},
\end{equation}
\endgroup
where $\mathcal{S}^{+}_{\dd(i)}$ denotes the set of all $\dd(i)\times \dd(i)$ positive definite real matrices for all $i \in [n]$. 

\bigskip
\noindent
When $\Q$ is the bipartite quiver from $(\ref{quiver-eqn-1})$ and $V=V_\F$, where $\F$ is an $n$-tuple of matrices with the same number of rows, we define the \emph{capacity of $(\F, \cc)$} to be
$$
\capa(\F, \cc):=\capa(V_\F, \cc).
$$

\smallskip
\item We say that $(V, \cc)$ is \emph{gaussian-extremizable} if there exist positive definite matrices $Y_i \in \mathcal{S}^{+}_{\dd(i)}$, $i \in [n]$, for which the infimum is attained in $(\ref{cap-eqn})$. If this is the case, we call such an $n$-tuple $(Y_1, \ldots, Y_n)$ a \emph{gaussian extremizer} for $(V, \cc)$.
\end{enumerate}
\end{definition}

\begin{rmk} \label{capacity-BL-operator-rmk} Assuming that the coordinates of $\cc$ are positive rational numbers, the positivity of $\capa(V, \cc)$ is equivalent to the positivity of the capacity of the so-called Brascamp-Lieb operator associated to the quiver datum $(V, \cc)$. For details, see \cite[Definition 4 $\&$ Lemma 8]{ChiDer-2019}.
\end{rmk}

\begin{remark}[\textbf{The capacity of a weighted matrix frame}]\label{cap-matrix-frame-def}
Let $\F=\{X_1, \ldots, X_n\}$ be an $n$-tuple of matrices  with $X_i \in \RR^{d \times d_i}$ for all $i \in [n]$.  Let $\Q$ be the quiver from $(\ref{quiver-eqn-1})$ and view $\F$ as the representation $\V_\F$ of $\Q$.  Then, for an $n$-tuple $\cc=(c_1, \ldots, c_n) \in \RR^n_{>0}$, we have that
$$
\capa(\F, \cc)=C  \inf \left \{ {\det\left( \sum_{i=1}^n e^{t_i}X_i X_i^T \right) \over e^{\langle t, \cc \rangle}} \;\middle|\; t=(t_1, \ldots, t_n) \in \RR^n \right \},
$$
where $C:=\prod_{i=1}^n c_i^{c_i}$. In particular this yields
$$
\log(\capa(\F,\cc))=f_\F(\cc)+\sum_{i=1}^n c_i \log(c_i).
$$
(The convention here is that $\log(\capa(\F, \cc))=-\infty$ when $\capa(\F, \cc)=0$.) So we can see that
\begin{itemize}
\item $\capa(\F, \cc)>0$ if and only if $f_\F(\cc)$ is finite;

\smallskip
\item $(V_\F, \cc)$ is gaussian-extremizable with gaussian extremizer $(\xi_1, \ldots, \xi_n) \in \RR^n_{>0}$ if and only if $f_\F(\cc)$ is finite and attained at $(t_1, \ldots, t_n)$ where
$$
e^{t_i}=c_i  \xi_i, \forall i \in [n].
$$
\end{itemize}
\end{remark}

The following result from \cite{ChiDer-2019} plays a key role in the proof of Theorem \ref{RIF-thm-1}.

\begin{theorem} \cite[Theorems 1 and 22]{ChiDer-2019}\label{Calin-Harm-thm-cap} Let $Q$ be a bipartite quiver with $n$ sink vertices,  $\dd$ a dimension vector of $Q$,  and $V \in \rep(Q,\dd)$ a $\dd$-dimensional representation of $Q$. Let $\cc=(c_1, \ldots, c_n) \in \QQ^n_{> 0}$ be an $n$-tuple of positive rational numbers such that 
$$
\sigma_{\cc} \cdot \dd=0.
$$

Then the following statements hold.

\begin{enumerate}
\item $V$ is $\cc$-semi-stable if and only if $\capa(V, \cc)>0$.

\smallskip
\item The following statements are equivalent
\begin{enumerate}[(i)]
\item $V$ is $\cc$-polystable;
\item $(V, \cc)$ can be transformed into a geometric BL quiver datum;
\item $(V, \cc)$ is gaussian-extremizable.
\end{enumerate}
\end{enumerate}
\end{theorem}

\begin{remark}\label{semi-stab-MF-polytime} We point out that Theorem \ref{Calin-Harm-thm-cap}{(1)} together with Gurvits's polynomial time algorithm \cite{GarGurOliWig-2017} shows that, given $\F$ and $\cc$, one can check whether $(\F, \cc)$ is a semi-stable frame in deterministic polynomial time. 
\end{remark}

\begin{corollary}\label{pmf-semi-stab-MF-coro} Let $\F=\{X_1, \ldots, X_n\}$ be an $n$-tuple of matrices  with $X_i \in \RR^{d \times d_i}$ for all $i \in [n]$.  Let $\Q$ be the quiver from $(\ref{quiver-eqn-1})$ and view $\F$ as the representation $\V_\F$ of $\Q$. 
\begin{enumerate}
\item If $(\F, \cc)$ is a semi-stable MF then it is a matrix frame;
\item If $(\F, \cc)$ is a PMF then it is a semi-stable MF. 
\end{enumerate}
\end{corollary}

\begin{proof} $(1)$ Assume that $(\F, \cc)$ is semi-stable. Then it follows from Theorem \ref{Calin-Harm-thm-cap}{(1)} that $\capa(\F, \cc)>0$ which clearly implies that $\det(\sum_{i=1}^n X_i  X_i^T)>0$. Thus $(\F, \cc)$ is a frame.

\bigskip
\noindent $(2)$ Assume that $(\F, \cc)$ is a PMF which is equivalent to $(V_\F, \cc)$ being a geometric BL quiver datum by Remark \ref{Geo-Bl-data-Parseval}. Then it follows that $V_\F$ is $\cc$-polystable by Theorem \ref{Calin-Harm-thm-cap}{(2)}. In particular, $V_\F$ is $\cc$-semi-stable since a direct sum of semi-stable representations is always semi-stable.
\end{proof}

\section{Matrix Radial Isotropy} Let $\F=\{X_1, \ldots, X_n\}$ be a matrix frame with $X_i \in \RR^{d \times d_i}$, $i \in [n]$, and let $\cc=(c_1, \ldots, c_n) \in \QQ_{>0}^n$ be positive rational weights such that  $\sum_{i=1}^n c_i=d$.  

Let $\Q$ be the quiver from $(\ref{quiver-eqn-1})$, $V_\F$ the representation of $\Q$ associated to $\F$, and $\sigma_{\cc}$ the weight of $\Q$ associated to $\cc$ via $(\ref{eqn-induced-wt})$.  Let $\dd$ be the dimension vector of $V_\F$ and note that $\dd$ is equal to one at the sink vertices $1, \ldots, n$ of $\Q$. Before we proceed with the proof of Theorem \ref{RIF-thm-1}, we make the following observations that allow us to switch back and forth between matrix frames and quiver representations.
\smallskip
\begin{enumerate}[(a)]
\item $\cc \in \rel(K_\F)$ if and only if $\sigma_{\cc} \in \rel(\Omega(V_\F))$ where $\Omega(V_\F)$ is the orbit cone of the representation $V_\F$ (see Remark \ref{polytope-vs-orbit-cone-rmk} and Definition \ref{orbit-cone-defn}).\\

\item $(\F, \cc)$ can be transformed into a RIF if and only if $(V_\F, \cc)$ can be transformed into a geometric BL quiver datum. Indeed, if $B \in \GL(d)$ transforms  $(\F, \cc)$ into a RIF then $(A  V_\F, \cc)$ is a geometric BL quiver datum where $A=\left((B^{-1})^T,  {1 \over \|B  X_1\|_F}, \ldots,  {1 \over \|B  X_n\|_F} \right) \in \GL(\dd)$. Conversely, if $C=(C_0, \lambda_1, \ldots, \lambda_n) \in \GL(d)\times \RR^{\times}\times \ldots \times \RR^{\times}$ transforms $(V, \cc)$ into a geometric BL quiver datum then $(C_0^{-1})^T$ transforms $(\F, \cc)$ into a RIF. \\

\item $f_\F(\cc)$ is finite and attained at some point of $\RR^n$ if and only if $(V_\F, \cc)$ is gaussian-extremizable by Remark \ref{cap-matrix-frame-def}.\\
\end{enumerate}

\begin{proof}[Proof of Theorem \ref{RIF-thm-1}] $(1)$ It follows from Remark \ref{cap-matrix-frame-def}, Theorem \ref{Calin-Harm-thm-cap}{(1)}, and Lemma \ref{orbit-poly-tuples-lemma} that
$$
f_\F(\cc) \text{ is finite} \Longleftrightarrow \capa(\F, \cc)>0 \Longleftrightarrow \cc \in K_{V_\F}=K_\F,
$$
and this shows that the domain of $\restr{f_\F}{\QQ_{>0}^n}$ is precisely $K_\F \cap \QQ_{>0}^n$.

\bigskip
\noindent
$(2)$ Let $\rep(\Q)^{ss}_{\cc}$ be the category of all $\cc$-semi-stable representations of $\Q$. It is immediate to see that simple objects of this category are precisely the $\cc$-stable representations of $Q$. Moreover, any object of $\rep(\Q)^{ss}_{\cc}$ has a Jordan-H{\"o}lder filtration whose factors are $\cc$-stable.  Thus, since $V_\F$ is $\cc$-semi-stable, we know that it has a filtration in $\rep(\Q)^{ss}_{\theta}$ whose factors, denoted by $V_1, \ldots, V_l$, are $c$-stable. This filtration gives rise to a $1$-parameter subgroup $\lambda: \RR^{\times} \to \GL(\dd)$ such that 
$$
\lim_{t \to 0}\lambda(t)V_\F \simeq \bigoplus_{i=1}^l V_l.
$$
In particular, this yields the $\cc$-polystable representation $\widetilde{V}:=\lim_{t \to 0}\lambda(t)V_\F$ which is a degeneration of $V_\F$ in the sense that $\widetilde{V} \in \overline{\GL(\dd)V_\F}$. Denoting by $\widetilde{\F}$ the frame corresponding to the representation $\widetilde{V}$, we get that $(\widetilde{\F}, \cc)$ can be transformed into a RIF by Theorem \ref{Calin-Harm-thm-cap}{(2)} and observation (b) above.

\bigskip
\noindent
$(3)$ The equivalence $(i) \Longleftrightarrow (ii)$ follows from Remark \ref{geom-BL-critical-reps}, the Quiver Radial Isotropy Theorem \ref{quiver-RadIso-thm}, and the observations (a) and (b) above.

\bigskip
\noindent
The equivalence $(ii) \Longleftrightarrow (iii)$ follows from Theorem \ref{Calin-Harm-thm-cap}{(2)} and the observations (b) and (c) above.
\end{proof}

\begin{remark}\label{rmk-classical-vs-matrix-RadIso} In what follows, we would like to point out some of the difficulties that one encounters when trying to use classical Radial Isotropy in order to establish our Matrix Radial Isotropy, even in particular cases. Let $\F=\{X_1,\ldots,X_n\}$ with $X_i\in\RR^{d\times d_i}$, $i\in[n]$, be an MF with $n>d$, and consider the frame (of vectors)
$$
\widetilde{\F}=\bigcup_{i \in [n]} \{\x_{i,1}, \ldots, \x_{i, d_i}\} \subseteq \RR^d,
$$
where $\x_{i,1}, \ldots, \x_{i, d_i}$ are the columns of $X_i$, $i \in [n]$.  The orbit polytope of the frame $\widetilde{\F}$ can be easily described as
\begingroup\makeatletter\def\f@size{10.5}\check@mathfonts
$$
K_{\widetilde{\F}}:=\left\{ \widetilde{\cc}=(c_{i,\ell})_{i \in [n], \ell \in [d_i]} \in \RR^N_{\geq 0}  \;\middle|\;
\begin{array}{l}
\sum_{i=1}^n \sum_{\ell=1}^{d_i }c_{i,\ell}=d, \text{~and~}\\
\\
\sum_{i \in [n]} \sum_{\ell \in I_i} c_{i,\ell} \leq \dim \Span(\x_{i,\ell} \mid i \in [n], \ell \in I_i),\\ 
\\
\text{for all subsets } I_i \subseteq [d_i], i\in [n]
\end{array} \right\}.
$$
\endgroup
Let us now assume that $\F$, equivalently $\widetilde{\F}$, is generic. Then it is immediate to check that 
$$
\widetilde{\cc}:=\left( \underbrace{{d \over d_1 n}, \ldots, {d \over d_1  n}}_{d_1 \text{ times}}, \ldots, \underbrace{{d \over d_n  n}, \ldots {d \over d_n  n}}_{d_n \text{ times}} \right) \in \rel(K_{\widetilde{\F}}).
$$
Thus, by (classical) Radial Isotropy, we know that there exists an invertible transformation $A \in \GL(d)$ such that
\begin{equation} \label{eqn-rmk-class-vs-matrix-RadIso}
\sum_{i=1}^n \sum_{\ell=1}^{d_i}{d \over d_i n} {(A  \x_{i,\ell})  (A \x_{i,\ell})^T \over \lvert \lvert A \x_{i,\ell} \rvert \rvert^2_F}=I_d.
\end{equation}
At this point it is natural to ask whether it is possible to find a matrix $A$ that satisfies $(\ref{eqn-rmk-class-vs-matrix-RadIso})$, and \emph{also}
\begin{equation}\label{eqn2-class-vs-matrix-RadIso}
\sum_{\ell=1}^{d_i}{d \over d_i n} {(A  \x_{i,\ell})  (A \x_{i,\ell})^T \over \lvert \lvert A \x_{i,\ell} \rvert \rvert^2_F}={d \over n} {(A  X_i) (A  X_i)^{T} \over \lvert \lvert A X_i \rvert \rvert^2_F}, \forall i \in [n].
\end{equation}
If true, this would show that $A$ transforms $(\F, \cc)$ into a RIF where $\cc=({d \over n}, \ldots, {d \over n})$. But it is easy to construct examples of generic MF for which $(\ref{eqn2-class-vs-matrix-RadIso})$ does not hold. For example, one can take
\begin{equation*}
\F=\left\{
X_1=\left(
\begin{matrix}
1 & 0\\
0 & 2
\end{matrix}
\right), 
X_2=\left(
\begin{matrix}
1 \\
-1
\end{matrix}
\right), 
X_3=\left(
\begin{matrix}
1\\
1
\end{matrix}
\right)
\right\}.
\end{equation*}
Then the identity matrix $I_2$ satisfies $(\ref{eqn-rmk-class-vs-matrix-RadIso})$ but not $(\ref{eqn2-class-vs-matrix-RadIso})$. In fact, as $V_{\widetilde{\F}}$ is a $\widetilde{\cc}$-stable representation with a one-dimensional space of endomorphisms, it follows that the matrices $A$ that satisfy $(\ref{eqn-rmk-class-vs-matrix-RadIso})$ are precisely those of the form $\lambda U  I_2=\lambda U$ with $U$ an orthogonal matrix and $\lambda \in \RR^{\times}$ (see for example \cite[Theorem 21]{ChiDer-2019}). Therefore, in this example, for any matrix $A$ that satisfies $(\ref{eqn-rmk-class-vs-matrix-RadIso})$, equation $(\ref{eqn2-class-vs-matrix-RadIso})$ does not hold.
\end{remark}

\section{The Paulsen Problem for matrix frames} \label{PaulsenProb-sec}

The proof strategy for Theorem \ref{Paulsen-MF-thm} is inspired by the work of Hamilton and Moitra in \cite{HamMoi-2018}. It relies on the following two lemmas. We postpone the proofs of these intermediate results until after the proof of Theorem \ref{Paulsen-MF-thm}. Throughout we assume that $\varepsilon<0.3$ which is what we mean by $\varepsilon$ being small. 

\begin{lemma}\label{1st-pert-lemma} Let $\F=\{X_1,\ldots,X_n\}$ with $X_i \in \RR^{d\times d_i}$, $i \in [n]$, be an $\varepsilon$-nearly equal-norm PMF  Then $\F$ can be perturbed such that the resulting frame $\G=\{Y_1,\ldots,Y_n\}$ with $Y_i \in \RR^{d\times d_i}$, $i \in [n]$,  satisfies the following properties:
\begin{enumerate}
\item $\G$ is a generic MF;
\item $\dist^2(\F, \G) \leq d\varepsilon$;
\item $\G$ is a $4\varepsilon$-nearly equal-norm PMF.
\end{enumerate}
\end{lemma}

We will also need the following upper bound. As explained in the proof of Theorem \ref{Paulsen-MF-thm} below, the passage from Lemma \ref{1st-pert-lemma} to Lemma \ref{2nd-inter-lemma} is via the Matrix Radial Isotropy Theorem \ref{RIF-thm-1}{(3)}. 

\begin{lemma}\label{2nd-inter-lemma} Let $\G=\{Y_1,\ldots,Y_n\}$ with $Y_i \in \RR^{d\times d_i}$, $i \in [n]$, be a $4\varepsilon$-nearly equal-norm PMF. Assume that there exists a diagonal matrix $M \in \RR^{d \times d}$ with positive weakly decreasing diagonal entries such that the frame $\{M  Y_1, \ldots, M  Y_n\}$ is a RIF with respect to the weight vector $\cc:=({d \over n}, \ldots, {d \over n}) \in \QQ_{>0}^n$. Let $\Z=\left\{ \sqrt{\frac{d}{n}}\frac{M  Y_1}{\|M  Y_1\|_F}, \ldots,  \sqrt{\frac{d}{n}}\frac{M  Y_n}{\|M  Y_n\|_F} \right\}$ be the induced equal-norm PMF.

Let $\gamma \leq \min\{1, \varepsilon\}$ be such that
\begin{equation} \label{condition on gamma}
(1-\gamma){d \over n} \leq \lvert \lvert Y_i \rvert \rvert^2_F \leq (1+\gamma){d \over n}, \forall i \in [n].
\end{equation}
Then
$$
\dist^2(\G, \Z)\leq 16 \varepsilon d^2+6 \gamma d^2.
$$
\end{lemma}

We are now ready to prove Theorem \ref{Paulsen-MF-thm}.

\begin{proof}[Proof of Theorem \ref{Paulsen-MF-thm}]
Let $\F=\{X_1,\ldots,X_n\}$ be an $\varepsilon$-nearly equal-norm PMF.  According to the proof of Lemma \ref{1st-pert-lemma}, we can choose ``perturbation" matrices $\Eta_1, \ldots, \Eta_n$ of arbitrarily small norm such that the matrix frame $\G=\{Y_1,\ldots,Y_n\}$, where
$$
Y_i=\sqrt{\frac{d}{n}}\frac{X_i}{\| X_i\|_F}+\Eta_i, i \in [n],
$$
satisfies properties $(1)-(3)$ in Lemma \ref{1st-pert-lemma}.  

By Lemma \ref{ex-gen-locally-semi-simple} and the Matrix Radial Isotropy Theorem \ref{RIF-thm-1}{(3)}, there exists an invertible matrix $A \in \GL(d)$ such that $A \cdot  \G:=\{A Y_1, \ldots, A Y_n\}$ is a RIF with respect to $\cc=({d \over n}, \ldots, {d \over n})$. Using the Singular Value Decomposition for $A$, we can write $A=U M  V^T$ where $U$ and $V$ are orthogonal matrices and $M$ is a diagonal matrix with positive weakly decreasing diagonal entries. Set
$$
\widetilde{\F}:=V^T \cdot  \F=\{V^T  X_1, \ldots, V^T  X_n\} \text{~and~} \widetilde{\G}:=V^T \cdot  \G=\{V^T  Y_1, \ldots, V^T  Y_n\}. 
$$
Then we have that
\begin{itemize}
\item $\dist^2(\widetilde{\F}, \widetilde{\G})=\dist^2(\F, \G) \leq \varepsilon d$;
\item $\widetilde{\G}$ is a $4\varepsilon$-nearly equal-norm PMF since $V \cdot \widetilde{\G}=\G$ a $4\varepsilon$-nearly equal-norm PMF;
\item $(M \cdot \widetilde{\G}, \cc)$ is a RIF since $U \cdot (M \cdot \widetilde{\G})=A \cdot \G$ is a RIF with respect to $\cc$.
\end{itemize}

Next, let $\gamma:=\max_{1\leq i\leq n}\frac{n}{d}(\|\Eta_i\|_F^2+2\|\Eta_i\|_F)$.  Since each $\|\Eta_i\|_F$ can be made arbitrarily small, we can assume that $\gamma \leq \min\{1, \varepsilon\}$.  Furthermore, one can easily check that $\gamma$ satisfies condition (\ref{condition on gamma}) in Lemma \ref{2nd-inter-lemma} applied to $\widetilde{\G}$. It now follows from Lemma \ref{2nd-inter-lemma} that there exists an equal-norm PMF $\mathcal{Z}$ such that $\dist^2(\widetilde{\G}, \mathcal{Z})\leq 16 \varepsilon d^2+6 \gamma d^2$.  

Using the triangle-like inequality
$$\|A-C\|_F^2\leq 2(\|A-B\|_F^2+\|B-C\|_F^2) \text{ for }A,B,C\in\RR^{d\times d_i},$$
we get that
\begin{align*}
    \text{dist}^2(\widetilde{\F},\mathcal{Z})
    &\leq 2(\text{dist}^2(\widetilde{\F},\widetilde{\G})+\text{dist}^2(\widetilde{\G},\mathcal{Z}))\\
    &\leq 2(\varepsilon d+16\varepsilon d^2+6\gamma d^2)\\
    &\leq 2(\varepsilon d+16\varepsilon d^2+6\varepsilon d^2)\text{ since }\gamma\leq\varepsilon\\
    &\leq 46\varepsilon d^2.
\end{align*}
Finally, setting $\W:=V \cdot \mathcal{Z}$, we obtain that $\dist^2(\F, \W)=\dist^2(\widetilde{\F},\mathcal{Z}) \leq 46\varepsilon d^2$.
\end{proof}

We end this section with the proofs of the intermediate lemmas. From this point on, we simply write $\|\cdot\|$ for $\|\cdot\|_F$.

\begin{proof}[Proof of Lemma \ref{1st-pert-lemma}]
We scale and perturb the $N$ columns of the matrix $[X_1 | \ldots | X_n] \in \RR^{d \times N}$ as follows: if $\x_{i,k}$ is the $k$-th column of $X_i$, let $y_{i,k}=\sqrt{\frac{d}{n}}\frac{\x_{i,k}}{\|X_i\|}+\eta_{i, k}$ where $\eta_{i, k} \in \RR^d$ are perturbations of arbitrary small norm chosen so that every $d$-subset of the vectors $\{y_{i,k}\}_{i\in[n],k\in[d_i]}$ forms a basis of $\RR^d$. Let us form the matrices $\Eta_i\in\RR^{d\times d_i}$ having columns $\eta_{i,k}, k \in [d_i]$, and define $Y_i:=\sqrt{\frac{d}{n}}\frac{X_i}{\| X_i\|}+\Eta_i$. We claim that the frame $\G:=\{Y_1,\ldots,Y_n\}$ satisfies the required properties.  

\smallskip
\noindent
$(1)$ The frame $\G$ is a generic MF by construction.

\smallskip
\noindent
$(2)$ Since $\F$ is an $\varepsilon$-nearly equal-norm PMF  with $\varepsilon < 0.3$ and $\|\Eta_i\|$ can be made as small as we wish, it follows that 
\begin{align*}
\|X_i-Y_i\|^2&=\|X_i\|^2+\tfrac dn-2\sqrt{\tfrac dn}\,\|X_i\|+\|\Eta_i\|^2+2\Bigl(\sqrt{\tfrac dn}\tfrac1{\|X_i\|}-1\Bigr)\langle X_i,\Eta_i\rangle_F\\
&\leq \left(\sqrt{\frac{d}{n}}-\|X_i\|\right)^2+\|\Eta_i\|^2+2\Bigl(\frac{1}{\sqrt{1-\varepsilon}}-1\Bigr)\|X_i\|\|\Eta_i\|\\
&\leq \left(\sqrt{\frac{d}{n}}-\|X_i\|\right)^2+\|\Eta_i\|^2+2\|\Eta_i\|\\
&\leq \Bigl(\sqrt{\tfrac dn}-\sqrt{(1-\varepsilon)\tfrac dn}\Bigr)^2+\|\Eta_i\|^2+2\|\Eta_i\|\\
&\leq \Bigl(\sqrt{\tfrac dn}-\sqrt{(1-\varepsilon)\tfrac dn}\Bigr)^2+\frac{(1-\sqrt{1-\varepsilon})\varepsilon d}{n}\\
&=\frac{d}{n}(1-2\sqrt{1-\varepsilon}+(1-\varepsilon)+(1-\sqrt{1-\varepsilon})\varepsilon )\leq\frac{d\varepsilon}{n}, \forall i \in [n].
\end{align*}
Thus, summing over $i\in[n]$, we obtain that $\dist^2(\F,\G)\leq\varepsilon d$. 

\smallskip
\noindent
$(3)$ To prove that $\G=\{Y_1,\cdots,Y_n\}$ is a $4\varepsilon$-nearly equal-norm PMF, we begin by imposing that $\|\Eta_i\|\leq\frac{\varepsilon}{2n}$ for all $i\in[n]$. Then it is immediate to check that $(1-4\varepsilon)\frac{d}{n}\leq\|Y_i\|^2\leq(1+4\varepsilon)\frac{d}{n},\ \forall i\in[n]$.

It remains to prove that $(1-4\varepsilon)I_d\preceq\sum_{i=1}^n Y_iY_i^T\preceq(1+4\varepsilon)I_d$, which is equivalent to proving that 
\[1-4\varepsilon\leq\sum_{i=1}^n\langle Y_i^Ty,Y_i^Ty\rangle\leq1+4\varepsilon, \forall y\in\RR^d \text{ with } \|y\|_2=1.\]

By assumption $\F$ is an $\varepsilon$-nearly equal-norm PMF and so we have
\begin{equation}\label{F is PMF}
\begin{split}
    \frac{1}{1+\varepsilon}&\leq\frac{d}{n}\frac{1}{\|X_i\|^2}\leq \frac{1}{1-\varepsilon},\forall i\in[n], \text{ and }\\ 1-\varepsilon&\leq\sum_{i=1}^n \langle X_i^Ty,X_i^Ty\rangle\leq1+\varepsilon, \forall y\in\RR^d \text{ with } \|y\|_2=1.
\end{split}
\end{equation}

Now let $y\in\RR^d$ be a unit vector. Then, using $(\ref{F is PMF})$ and keeping in mind that $\|\Eta_i\|\leq\frac{\varepsilon}{2n}$ for all $i\in[n]$, we get
\begin{align*}
    \sum_{i=1}^n \langle Y_i^Ty,Y_i^Ty\rangle&=\sum_{i=1}^n \left(\frac{d}{n}\frac{1}{\|X_i\|^2}\langle X_i^Ty,X_i^Ty\rangle+\langle \Eta_i^Ty,\Eta_i^Ty\rangle+2\sqrt{\frac{d}{n}}\frac{1}{\| X_i\|}\left\langle X_i^T y,\Eta_i^Ty\right\rangle\right)\\
    &\leq \frac{1+\varepsilon}{1-\varepsilon}+\sum_{i=1}^n\left(\|\Eta_i\|^2\| y\|_2^2+2\sqrt{\frac{d}{n}}\frac{1}{\|X_i\|}\|X_i\|\|\Eta_i\|\| y\|_2^2\right)\\
    &\leq\frac{1+\varepsilon}{1-\varepsilon}+\sum_{i=1}^n \left(\frac{\varepsilon^2}{4n^2}+\frac{\varepsilon}{n}\right)=\frac{1+\varepsilon}{1-\varepsilon}+\frac{\varepsilon^2}{4n}+\varepsilon\leq 1+4\varepsilon \text{ for small }\varepsilon.
\end{align*}

Similarly, we obtain that
\begin{align*}
    \sum_{i=1}^n \langle Y_i^Ty, Y_i^Ty\rangle&=\sum_{i=1}^n \left(\frac{d}{n}\frac{1}{\|X_i\|^2}\langle X_i^Ty,X_i^Ty\rangle+\langle \Eta_i^Ty,\Eta_i^Ty\rangle+2\sqrt{\frac{d}{n}}\frac{1}{\| X_i\|}\left\langle X_i^T y,\Eta_i^Ty\right\rangle\right)\\
    &\geq \frac{1-\varepsilon}{1+\varepsilon}-\sum_{i=1}^n 2\sqrt{\frac{d}{n}}\frac{1}{\|X_i\|}\|X_i\|\|\Eta_i\|\|y\|_2^2 \\
    &\geq \frac{1-\varepsilon}{1+\varepsilon}-\sum_{i=1}^n \frac{\varepsilon}{n}=\frac{1-\varepsilon}{1+\varepsilon}-\varepsilon\geq 1-4\varepsilon.
\end{align*}

\end{proof}

The proof of Lemma \ref{2nd-inter-lemma} requires the following result from \cite{HamMoi-2018}. Let $\mathcal{T}:\RR^d\times \RR^d \to \RR$ be the linear operator defined by  
$$
\mathcal{T}(v,u):=\sum_{\ell=1}^d \ell(u_\ell-v_\ell), \forall v, u \in \RR^d.
$$

\noindent
Given two vectors $v,u\in\RR^d$, we say that $v$ \emph{majorizes} $u$, denoted $v\succeq u$, if $\sum_{\ell=1}^d v_\ell=\sum_{\ell=1}^d u_\ell$ and $\sum_{\ell=1}^j v_\ell\geq\sum_{\ell=1}^j u_\ell$ for all $1\leq j< d$. 

\begin{lemma} If $v, u \in \RR^d$ are two vectors such that $v\succeq u$ then
\begin{equation} \label{maj-ineq}
\mathcal{T}(v,u) \geq {1 \over 2} \|u-v\|_1.
\end{equation}
\end{lemma}

\begin{proof} We proceed by induction on $d$.  The statement is clearly true when $d=1$. Let us now assume that $(\ref{maj-ineq})$ holds for all pairs $(v, u) \in \RR^d\times \RR^d$ with $v\succeq u$.  

Let $x$ and $y$ be two vectors in $\RR^{d+1}$ such that $x\succeq y$. We will show that $(\ref{maj-ineq})$ holds for $(x, y)$, as well. Set
$$
c:=\sum_{\ell=1}^d (x_{\ell}-y_{\ell})=y_{d+1}-x_{d+1}\geq 0,
$$
and let us consider the vectors in $\RR^d$
$$
v:=(x_1, \ldots, x_d) \text{~and~} u:=(y_1, \ldots, y_{d-1}, y_d+c).
$$

\noindent
We then have that 
\begin{itemize}
\item $v \succeq u$ since $\sum_{\ell=1}^dx_\ell=y_1+\ldots y_{d-1}+(y_d+c)$ by the definition of $c$ and, furthermore, $\sum_{\ell=1}^j x_\ell\geq\sum_{\ell=1}^j y_\ell$ for all $1\leq j< d-1$ as $x\succeq y$;\\
\item $\mathcal{T}(x,y)=\sum_{\ell=1}^{d-1} \ell(y_{\ell}-x_{\ell})+d(y_d-x_d)+(d+1)(y_{d+1}-x_{d+1})=\mathcal{T}(v, u)+c$ as $y_{d+1}-x_{d+1}=c$. 
\end{itemize}
 
\smallskip
\noindent 
Using the induction hypothesis, we get that
$$
\mathcal{T}(x,y) \geq {1 \over 2}\|u-v\|_1+c={1 \over 2} \left( \sum_{\ell=1}^{d-1}|y_{\ell}-x_{\ell}|+|y_d-x_d+c|+|c|+|y_{d+1}-x_{d+1}|\right) \geq {1 \over 2}\|x-y\|_1,
$$
and this finishes the proof.
\end{proof}

\begin{proof}[Proof of Lemma \ref{2nd-inter-lemma}] Let us define the helper frame
\[\WW=\{\ww_1,\ww_2,\cdots,\ww_n\}\text{ where }\ww_i=\|Y_i\|\left(\frac{MY_i}{\|MY_i\|}\right)\in\RR^{d\times d_i},\ \forall i\in[n].\]

\noindent
Using the triangle-like inequality, we have that
$$
\dist^2(\G,\Z)\leq 2\ \dist^2(\G,\WW)+2\ \dist^2(\WW,\Z).
$$
Therefore, to prove the desired upper bound for $\dist^2(\G,\Z)$, we will find suitable upper bounds for $\dist^2(\WW,\Z)$ and $\dist^2(\G,\WW)$. Using $(\ref{condition on gamma})$, we immediately get the first upper bound
\begin{align*}
    \dist^2(\WW,\Z)
    &=\sum_{i=1}^n \left\|\|Y_i\|\left(\frac{MY_i}{\|MY_i\|}\right)-\sqrt{\frac{d}{n}}\left(\frac{MY_i}{\|MY_i\|}\right)\right\|^2=\sum_{i=1}^n\frac{d}{n}\left|\frac{\|Y_i\|}{\sqrt{\frac{d}{n}}}-1\right|^2\\
    &\leq \sum_{i=1}^n\frac{d}{n}\max\{(\sqrt{1+\gamma}-1)^2, (\sqrt{1-\gamma}-1)^2\}=d (\sqrt{1-\gamma}-1)^2\leq d^2\gamma.
\end{align*}

In what follows we will find an upper bound for $\dist^2(\G,\WW)$. Using the assumption that $\Z$ is an equal-norm PMF and $(1-\gamma)\frac{d}{n}\leq\|Y_i\|^2,\ \forall i\in[n]$, we obtain
\begin{align*}
    \sum_{i=1}^n  \ww_i\ww_i^T&=\sum_{i=1}^n\|Y_i\|^2\left(\frac{MY_i}{\|MY_i\|}\right)\left(\frac{MY_i}{\|MY_i\|}\right)^T\\
    &\succeq\sum_{i=1}^n(1-\gamma)\frac{d}{n}\left(\frac{MY_i}{\|MY_i\|}\right)\left(\frac{MY_i}{\|MY_i\|}\right)^T=(1-\gamma) I_d.
\end{align*}
This in turn implies that 
\begin{equation} \label{ineq-3rd-aux-lemma}
\sum_{i=1}^n\|\ww_i^Ty\|_2^2=\sum_{i=1}^n\langle\ww_i^Ty,\ww_i^Ty\rangle_2\geq1-\gamma,
\end{equation}
for all unit vectors $y \in \RR^d$.  Now, let us write $Y_i=((s_i)_{\ell k})_{\ell\in[d],k\in[d_i]}$ and $\ww_i=((t_i)_{\ell k})_{\ell\in[d],k\in[d_i]}$ for every $i\in[n]$.  Using $(\ref{ineq-3rd-aux-lemma})$ with $y$ the $\ell$-th standard basis vector in $\RR^d$, we obtain
\begin{equation}\label{alpha}
    \sum_{i=1}^n \sum_{k=1}^{d_i} (t_i)^2_{\ell k}\geq 1-\gamma.
\end{equation}
Similarly, since $\G$ is a $4\varepsilon$-nearly equal-norm PMF, it follows that

\begin{equation}\label{beta}
\sum_{i=1}^n \sum_{k=1}^{d_i} (s_i)^2_{\ell k}\leq 1+4\varepsilon.
\end{equation}

Next, for every $i\in[n]$, we define $a^i,b^i\in\RR^d$ to be the vectors with coordinates:
\[a^i_\ell=\sum_{k=1}^{d_i}(t_i)^2_{\ell k}\text{ and }b^i_\ell=\sum_{k=1}^{d_i}(s_i)^2_{\ell k}, \forall \ell\in[d].\]

Then
\begin{align*}
    \dist^2(\G,\WW) &=\sum_{i=1}^n\| Y_i-\ww_i\|^2=\sum_{i=1}^n\sum_{\ell=1}^d\sum_{k=1}^{d_i} |(s_i)_{\ell k}-(t_i)_{\ell k}|^2\\
    &\stackrel{(i)}{\leq}\sum_{i=1}^n\sum_{\ell=1}^d \left|\sum_{k=1}^{d_i}(s_i)^2_{\ell k}-\sum_{k=1}^{d_i}(t_i)^2_{\ell k}\right|= \sum_{i=1}^{n} \|b^i-a^i\|_1\\
    &\stackrel{(ii)}{\leq} 2 \left( \sum_{i=1}^{n} \mathcal{T}(a^i,b^i) \right)=2 \mathcal{T}\left(\sum_{i=1}^{n}a^i,\sum_{i=1}^{n}b^i\right)=2 \left( \sum_{\ell=1}^d \ell\left(\sum_{i=1}^{n} \sum_{k=1}^{d_i}(s_i)^2_{\ell k}-\sum_{i=1}^{n} \sum_{k=1}^{d_i}(t_i)^2_{\ell k}\right) \right)\\
    &\leq 2 \left( \sum_{\ell=1}^d \ell ((1+4\varepsilon)-(1-\gamma))\right) \text{ using ($\ref{alpha}$) and ($\ref{beta}$)}\\
    &=(4\varepsilon+\gamma)d(d+1)\leq8\varepsilon d^2+2 \gamma d^2.
\end{align*}
For $(i)$, we first write
$$ (t_i)_{\ell k}=\lambda_{\ell} \frac{\|Y_i\|}{\|MY_i\|}(s_i)_{\ell k}$$ where the $\lambda_{\ell}$ are the positive diagonal entries of $M$.  Then for all  $i\in[n]$ and $\ell\in[d]$ we have

\begin{align*}
\sum_{k=1}^{d_i} |(s_i)_{\ell k}-(t_i)_{\ell k}|^2&=\left(1-\frac{\|Y_i\|}{\|MY_i\|}\lambda_\ell\right)^2 \left( \sum_{k=1}^{d_i}(s_i)^2_{\ell k} \right)\\
&\leq\left|1-\frac{\|Y_i\|^2}{\|MY_i\|^2}\lambda_\ell^2\right| \left( \sum_{k=1}^{d_i}(s_i)^2_{\ell k} \right) =\left|\sum_{k=1}^{d_i}(s_i)^2_{\ell k}-\sum_{k=1}^{d_i}(t_i)^2_{\ell k}\right|.
\end{align*}

\smallskip
\noindent
For $(ii)$, we claim that $a^i\succeq b^i, \forall i\in[n]$, and hence $\|b^i-a^i\|_1 \leq 2\mathcal{T}(a^i, b^i)$.

\smallskip
\noindent
\textbf{Claim:} The vector $a^i$ majorizes $b^i$ for all $i \in [n]$.

\begin{proof}[Proof of \textbf{Claim}]
We need to show that
\begin{equation}\label{eqn-1-claim-main-lemma2}
\sum_{\ell=1}^d a^i_\ell=\sum_{\ell=1}^d b^i_\ell,
\end{equation}
and
\begin{equation}\label{eqn-2-claim-main-lemma2}
\sum_{\ell=1}^j a^i_\ell\geq\sum_{\ell=1}^j b^i_\ell, \forall 1 \leq j<d.
\end{equation}

\noindent
Since $\|\ww_i\|^2=\|Y_i\|^2$ for all $i\in[n]$, we get that 
$$\sum_{k=1}^{d_i}\sum_{\ell=1}^d (t_i)_{\ell k}^2=\sum_{k=1}^{d_i}\sum_{\ell=1}^d (s_i)^2_{\ell k},$$ which proves $(\ref{eqn-1-claim-main-lemma2})$.  Next recall that $(t_i)_{\ell k}^2=\frac{\|Y_i\|^2}{\|MY_i\|^2}\lambda_\ell^2(s_i)_{\ell k}^2$. Therefore, for every $1\leq j<d$, we have
\begin{align*}
    \frac{\sum_{\ell=1}^j a^i_\ell}{\sum_{\ell=j+1}^d a^i_\ell}&=\frac{\sum_{\ell=1}^j \sum_{k=1}^{d_i}(t_i)^2_{\ell k}}{\sum_{\ell=j+1}^d \sum_{k=1}^{d_i}(t_i)^2_{\ell k}}=\frac{\sum_{k=1}^{d_i}\sum_{\ell=1}^j \lambda_\ell^2(s_i)^2_{\ell k}}{\sum_{k=1}^{d_i}\sum_{\ell=j+1}^d \lambda_\ell^2(s_i)_{\ell k}^2}\\
    &\geq \frac{\sum_{k=1}^{d_i}\sum_{\ell=1}^j \lambda_j^2(s_i)_{\ell k}^2}{\sum_{k=1}^{d_i}\sum_{\ell=j+1}^d \lambda_j^2(s_i)_{\ell k}^2}=\frac{\sum_{\ell=1}^j \sum_{k=1}^{d_i}(s_i)^2_{\ell k}}{\sum_{\ell=j+1}^d \sum_{k=1}^{d_i}(s_i)^2_{\ell k}}=\frac{\sum_{\ell=1}^j b^i_\ell}{\sum_{\ell=j+1}^d b^i_\ell},
\end{align*}
where the inequality follows from the fact that $\lambda_1\geq\lambda_2\geq\cdots\geq\lambda_d>0.$ This inequality combined with $(\ref{eqn-1-claim-main-lemma2})$ clearly implies $(\ref{eqn-2-claim-main-lemma2})$, finishing the proof of our claim.
\end{proof}

Putting everything together, we finally get
\begin{align*}
    \dist^2(\G,\Z)&\leq 2\ \dist^2(\G,\WW)+2\ \dist^2(\WW,\Z) \leq 2(8\varepsilon d^2+2\gamma d^2)+2(d^2\gamma)=16\varepsilon d^2+6\gamma d^2,
\end{align*}
and this completes the proof of the lemma.
\end{proof}

\section{Constructive aspects of Matrix Radial Isotropy}

Let $\F=\{X_1,\ldots,X_n\}$ with $X_i\in\RR^{d\times d_i}$, $i\in[n]$, be an MF, and $\Phi_\F:\RR^n\longrightarrow\RR$ the function defined by
\[\Phi_\F(t)=\log(\det(Q(t)))\]
where $Q(t)=\sum_{i=1}^n e^{t_i}X_iX_i^T$  for all $t=(t_1,\ldots,t_n)\in\RR^n$.  (Whenever $\F$ is understood from the context, we will simply write $\Phi$ for $\Phi_\F$.) We will also need the following index set
\begingroup\makeatletter\def\f@size{11}\check@mathfonts
$$
\s=\left \{ (I, (S_i)_{i \in I}) \;\middle|\;
\begin{array}{l}
\emptyset \neq I \subseteq [n],\\
\emptyset \neq S_i \subseteq [d_i], \forall i \in I, \text{~and}\\
\sum_{i \in I} |S_i|=d
\end{array}    \right \}.
$$
\endgroup
For $S=(I, (S_i)_{i \in I}) \in \s$ and $i \in [n]$, we write $i \in S$ to mean that $i \in I$. Furthermore, we define
$$
\Delta_S:=\det\left( \sum_{\ell \in S} X[S_{\ell}]\cdot X[S_{\ell}]^T \right),
$$
where $X[S_{\ell}]$ consists of the columns of the matrix $X_{\ell}$ indexed by the elements of the subset $S_{\ell} \subseteq [d_{\ell}]$ for each $\ell \in S$.

\begin{lemma}\label{lemma1-CARI}
Keep the same notation as above. Then, for every $t=(t_1, \ldots, t_n) \in \RR^n$,
\begin{equation}\label{eqn1}
    \frac{\partial\Phi}{\partial t_i}(t)=\frac{\displaystyle\sum_{S\in \s,i\in S}|S_i| e^{\sum_{\ell\in S}|S_\ell|t_\ell} \Delta_S}{\displaystyle\sum_{S\in\s}e^{\sum_{\ell \in S} |S_\ell|t_\ell} \Delta_S},
\end{equation}
and
\begin{equation}\label{eqn2}
    \frac{\partial \Phi}{\partial t_i}(t)=e^{t_i}\|Q^{-\frac{1}{2}}(t) X_i\|^2, \forall i \in [n].
\end{equation}
\end{lemma}

\begin{proof} We begin by writing
\begin{equation}\label{eqn3}
    \det(Q(t))=\sum_{S\in\s} \det\left(\sum_{\ell\in S}e^{t_\ell}X[S_\ell]\cdot X[S_\ell]^T\right)=\sum_{S\in\s} e^{\sum_{\ell\in S}|S_\ell|t_\ell} \Delta_S.
\end{equation}

\noindent
The first equality follows from the Cauchy-Binet formula applied to $\det(\mathfrak{X} \mathfrak{X}^T)$ where $\mathfrak{X}:=[e^{{t_1\over 2}}X_1|...|e^{{t_n\over 2}}X_n]$ is the matrix whose columns are the columns of the matrices $e^{t_i \over 2}X_i$, $i \in [n]$. For the second equality note that each $\sum_{\ell\in S}e^{t_\ell}X[S_\ell] \cdot X[S_\ell]^T$ can be written as a sum of $d$ matrices of the form $e^{t_\ell}x  x^T$ with $x \in \RR^d$ a column of $X[S_\ell]$. This combined with the general formula $\det \left( \sum_{i=1}^d c_i x_i  x_i^T \right)=\left( \prod_{i=1}^d c_i\right) \det \left( \sum_{i=1}^d x_i  x_i^T \right)$ that holds for any $d$ vectors $x_1, \ldots, x_d$ in $\RR^d$ and positive coefficients $c_1, \ldots, c_d$, yields the second equality in $(\ref{eqn3})$.

It now follows from $(\ref{eqn3})$ that
\[\frac{\partial\Phi}{\partial t_i}(t)=\frac{1}{\det(Q(t))}\left(\sum_{\substack{S\in \s\\i\in S}} |S_i| e^{\sum_{\ell\in S}|S_\ell|t_\ell}\Delta_S\right),\]
and this proves $( \ref{eqn1})$. 

\smallskip
\noindent
To prove $(\ref{eqn2})$, we write
\[\frac{\partial \Phi}{\partial t_i}(t)=\frac{1}{\det(Q(t))}\frac{\partial}{\partial t_i}(\det(Q(t)))=\lim_{\varepsilon\to 0}\frac{\det(Q(t+\varepsilon \delta_i))-\det(Q(t))}{\varepsilon\det(Q(t))},\]
where $\delta_i$ is the $i$th standard basis vector of $\RR^n$. Next, we have that
\begin{align*}
    Q(t+\varepsilon\delta_i)&=\sum_{\ell=1}^n e^{t_\ell+\varepsilon \delta_{i,\ell}}X_\ell X_\ell^T=\sum_{\substack{\ell=1\\\ell\neq i}}^n e^{t_\ell}X_\ell X_\ell^T +e^{t_i+\varepsilon}X_iX_i^T\\
    &=\sum_{\substack{\ell=1\\\ell\neq i}}^n e^{t_\ell}X_\ell X_\ell^T +e^{t_i}X_iX_i^T\left(1+\varepsilon+\frac{\varepsilon^2}{2}+\ldots\right)=\sum_{\ell=1}^n e^{t_\ell}X_\ell X_\ell^T +\varepsilon e^{t_i}X_iX_i^T+\varepsilon^2Y_i,
\end{align*}
where $Y_i:=e^{t_i}X_iX_i^T(\frac{1}{2}+\frac{\varepsilon}{3!}+\ldots)$. Thus we get that
\[\det(Q(t+\varepsilon \delta_i))=\det(Q(t)+\varepsilon e^{t_i}X_iX_i^T)+\varepsilon^2 f_i(\varepsilon)\] for a suitable continuous function $f_i(\varepsilon)$. This yields
\begin{equation}\label{eqn4}
\begin{split}
    \frac{\partial \Phi}{\partial t_i}(t)&=\lim_{\varepsilon\to 0}\frac{\det(Q(t)+\varepsilon e^{t_i}X_iX_i^T)-\det(Q(t))}{\varepsilon\det(Q(t))}\\
    &=\frac{1}{\det(Q(t))}\lim_{\varepsilon\to 0}\frac{\det(Q(t)+\varepsilon e^{t_i}X_iX_i^T)-\det(Q(t))}{\varepsilon}.
\end{split}
\end{equation}
Finally, using the following differentiation formula (see \cite[Appendix A.2]{Art-AviKapHai-2020})
\begin{equation}\label{eqn5}
    \frac{\partial}{\partial B}\det(A):=\lim_{\varepsilon\to0}\frac{\det(A+\varepsilon B)-\det(A)}{\varepsilon}=\det(A) \Tr(A^{-1}B),
\end{equation}
which holds for all square matrices $A$ and $B$ with $A$ invertible, it follows from $(\ref{eqn4})$ and $(\ref{eqn5})$ that
\begin{align*}
    \frac{\partial\Phi}{\partial t_i}(t)&=\Tr(Q(t)^{-1}e^{t_i}X_iX_i^T)=e^{t_i} \Tr(Q(t)^{-\frac{1}{2}}X_i(Q(t)^{-\frac{1}{2}} X_i)^T)=e^{t_i}\|Q(t)^{-\frac{1}{2}}X_i\|,
\end{align*}
and this completes the proof.
\end{proof}

To prove Theorem \ref{RIF-thm-2} we will also need the following result from \cite{ChiDer-2019}.

\begin{theorem}[\textbf{Gaussian-extremizers}] \cite[Theorems 22]{ChiDer-2019}\label{Calin-Harm-thm} Let $Q$ be a bipartite quiver with set of sink vertices $Q^{-}_0=\{1, \ldots, n\}$,  $\dd$ a dimension vector of $Q$,  and $\cc=(c_1, \ldots, c_n) \in \QQ^n_{>0}$ an $n$-tuple of positive rational numbers such that 
$$
\sigma_{\cc} \cdot \dd=0.
$$
Let $V \in \rep(Q, \dd)$ be a $\dd$-dimensional representation and consider the real algebraic set 
$$
\BL_{\cc}(V)=\{A \in \GL(\dd) \mid (A \cdot V,\cc) \text{~is a geometric BL quiver datum}\}.
$$
If $\BL_{\cc}(V)$ is not empty then $(V, \cc)$ is gaussian-extremizable and its gaussian extremizers are the $n$-tuples of matrices
$$
(A(i)^T  A(i))_{i \in [n]} \text{~with~}A=(A(x))_{x\in Q_0} \in \BL_{\cc}(V). 
$$
\end{theorem}

We are now ready to prove Theorem \ref{RIF-thm-2}.

\begin{proof}[Proof of Theorem \ref{RIF-thm-2}]
For the implication $(1)\Longrightarrow(2)$, we know that $\Phi(t)-\langle t,c\rangle$ has a minimum at $t^*=(t^*_1, \ldots, t^*_n)$ which implies that 
\[\frac{\partial\Phi}{\partial t_i}(t^*)=c_i,\forall i\in[n].\]
This combined with $(\ref{eqn1})$ in Lemma \ref{lemma1-CARI}, proves that
\[(e^{t_1^*},\ldots,e^{t_n^*})\in\mathbb{V}(\F,c).\]

\smallskip
\noindent
For the other implication $(2)\Longrightarrow (1)$, since $(e^{t_1^*},\ldots,e^{t_n^*})\in\mathbb{V}(\F,c)$, we get via Lemma $\ref{lemma1-CARI}$ that
\begin{equation}\label{eqn6}
e^{t_i^*}\|Q^{-\frac{1}{2}(t^*)}X_i\|^2=c_i,\forall i\in[n].
\end{equation}

\noindent
Thus we obtain that
\begin{align*}
    I_d=Q^{-\frac{1}{2}}(t^*) Q(t^*) Q^{-\frac{1}{2}}(t^*)
    &=\sum_{i=1}^n e^{t_i^*} Q^{-\frac{1}{2}}(t^*) X_i X_i^T Q^{-\frac{1}{2}}(t^*)\\
    &\stackrel{(\ref{eqn6})}{=}\sum_{i=1}^n \frac{c_i}{\|Q^{-\frac{1}{2}}(t^*)X_i\|^2}\left(Q^{-\frac{1}{2}}(t^*) X_i \right) \left(Q^{-\frac{1}{2}}(t^*)  X_i\right)^T.
\end{align*}
This shows that $Q^{-\frac{1}{2}}(t^*)$ transforms $(\F,\cc)$ into a RIF. But this is equivalent to saying that $(A \cdot V_\F,\cc)$ is a geometric BL quiver datum over our quiver $\Q$ from $(\ref{quiver-eqn-1})$ where
\[A:=\left(Q^{\frac{1}{2}}(t^*),\frac{1}{\|Q^{-\frac{1}{2}}(t^*)X_1\|},\ldots,\frac{1}{\|Q^{-\frac{1}{2}}(t^*)X_n\|}\right)\in\GL(\dd).\]

\noindent
At this point, we can use Theorem \ref{Calin-Harm-thm} to conclude that 
$(\xi_1,\ldots,\xi_n)$ with  
$$
\xi_i=A(i)^T A(i)=\frac{1}{\|Q^{-\frac{1}{2}}(t^*)X_i\|^2}, i\in[n],
$$ 
is a gaussian-extremizer for $(V_\F,\cc)$. This is further equivalent to saying that $f_\F(\cc)$ is finite and attained at $(t_1,\ldots,t_n)$ where
\[e^{t_i}=c_i\xi_i=\frac{c_i}{\|Q^{-\frac{1}{2}}(t^*)X_i\|^2}\stackrel{(\ref{eqn6})}{=}e^{t_i^*},\forall i\in[n].\]
This shows that $f_\F(\cc)$ is indeed attained at $t^*$, and this completes the proof.
\end{proof}

\begin{rmk} When $d_1=\ldots=d_n=1$, Theorem \ref{RIF-thm-2} was first proved in \cite{Bar-1998} (see also \cite[Appendix A]{Art-AviKapHai-2020}). We point out that the key ingredients used in \cite[Lemma A.4]{Art-AviKapHai-2020} do not seem to extend to our set-up, making Theorem \ref{Calin-Harm-thm} indispensable for the general case of matrix frames.
\end{rmk}

\section{Quiver Radial Isotropy and $\sigma$-critical quiver representations}
Throughout this section, unless otherwise specified, the only topology we work with is the Euclidean topology.

Let $Q=(Q_0,Q_1,t,h)$ be a connected acyclic quiver, $\dd \in \ZZ_{\geq 0}^{Q_0}$ a dimension vector, and $V \in \rep(Q,\dd)$ a $\dd$-dimensional representation of $Q$.

\begin{definition}[\textbf{The orbit cone of a representation}]\label{orbit-cone-defn} The \emph{orbit cone} of $V$ is the rational convex polyhedral cone defined by
$$
\Omega(V):=\{\sigma \in \RR^{Q_0} \mid \sigma \cdot \ddim \V=0 \text{~and~} \sigma \cdot \ddim \V' \leq 0,~ \forall \V' \subseteq \V\}.
$$
\end{definition}

\begin{rmk} The terminology ``orbit cone" is justified by the fact that the faces of $\Omega(V)$ are precisely the orbit cones of the form $\Omega(W)$ with $W \in \overline{\GL(\dd)V}$ (see for example \cite{CC5}). 
\end{rmk}

\noindent
Our goal is to give a representation-theoretic interpretation (see Theorem \ref{quiver-RadIso-thm} below) of the relative interior points of these cones in terms of the so-called critical quiver representations. 

\begin{definition}[\textbf{Critical representations}] \label{crit-rep-defn} Let $\sigma \in \ZZ^{Q_0}$ be a weight of $Q$. 
\begin{enumerate}[(i)]
\item We say that $V$ is a \emph{$\sigma$-critical representation} if $V$ satisfies the following matrix equations \\
\begin{equation*}
\sum_{a \in Q_1, ta=x} V(a)^T \cdot V(a)-\sum_{a \in Q_1, ha=x} V(a) \cdot V(a)^T=\sigma(x) \Id_{\dd(x)}, \forall x \in Q_0. 
\end{equation*}

\item We say that a transformation $A \in \GL(\dd)$ \emph{puts $V$ in $\sigma$-critical position} if $A \cdot V$ is a $\sigma$-critical representation.
\end{enumerate}
\end{definition}

\begin{remark}[\textbf{Geometric BL quiver data and critical representations}]\label{geom-BL-critical-reps} We point out that after appropriately scaling representations, geometric BL quiver data become critical quiver representations. Indeed, assume that $Q$ is a bipartite quiver with $n$ sink vertices and let $V \in \rep(Q, \dd)$ be a $\dd$-dimensional representation of $Q$. Let $\cc \in \QQ^n_{>0}$ be a rational vector and consider the weight $\sigma_{\cc}$ defined via $(\ref{eqn-induced-wt})$. Let $W \in \rep(Q,\dd)$ be the representation defined by 
$$W(a):=\sqrt{-\sigma_{\cc}(ha)}V(a), \forall a \in Q_1.
$$ 
It is then immediate to see that $(V,  \cc)$ is a geometric BL quiver datum if and only if  $W$ is a $\sigma_{\cc}$-critical representation. 

Consequently, we get that a quiver datum $(V, \cc)$ can be transformed into a geometric BL quiver datum if and only if $V$ can be put into $\sigma_{\cc}$-critical position.
\end{remark}

Recall that, given a weight $\sigma \in \ZZ^{Q_0}$, we say that $V$ is $\sigma$-polystable if $V$ is a direct sum of $\sigma$-stable representations. We have the following important result which was first proved by King in \cite{K} over the field of complex numbers. As explained in \cite{ChiDer-2019}, it also holds over the field of real numbers. In what follows, $\GL(\dd)_{\sigma}$ denotes the kernel of the character induced by $\sigma$, i.e.
$$
\GL(\dd)_{\sigma}=\left \{A=(A(x))_{x \in Q_0} \in \GL(\dd) \;\middle|\; \prod_{x \in Q_0} \det(A(x))^{\sigma(x)}=1 \right \}.
$$

\begin{prop}\label{poly-stable-prop-app}(see \cite[Proposition 14 and Remark 15]{ChiDer-2019}) Let $\sigma \in \ZZ^{Q_0}$ be a weight such that $V \in \rep(Q, \dd)$ is $\sigma$-semi-stable. Then the following statements are equivalent:
\begin{enumerate}
\item $V$ is $\sigma$-polystable;

\item the $\GL(\dd)_{\sigma}$-orbit of $V$ is closed in $\rep(Q,\dd)$;

\item $V$ can be put in $\sigma$-critical position.
\end{enumerate}
\end{prop}

Recall that a representation $V$ is said to be locally semi-simple if there exists a weight $\sigma_0 \in \ZZ^{Q}$ such that $V$ is a $\sigma_0$-polystable representation. For a rational convex polyhedral cone $C$, we denote its relative interior by $\rel(C)$.

Our next result, which plays a key role in the proof of Theorem \ref{RIF-thm-1}, can be viewed as a far reaching generalization of Barthe's theorem on vectors in radial isotropic position.

\begin{theorem}[\textbf{Quiver Radial Isotropy}]\label{quiver-RadIso-thm} Assume that $V \in \rep(Q, \dd)$ is a locally semi-simple representation and let $\sigma \in \ZZ^{Q_0}$ be a weight of $Q$. Then the following statements are equivalent
\begin{enumerate}
\item $\sigma \in \rel(\Omega(V))$;

\item $V$ can be put in $\sigma$-critical position.
\end{enumerate}
\end{theorem}

To prove this theorem, we require several auxiliary results. We begin by recalling the following simple, well-known folklore lemma. It is implicit in the work of King on moduli spaces of quiver representations \cite{K}.

\begin{lemma} \label{lemma-1-app} Let $V \in \rep(Q, \dd)$ be a $\dd$-dimensional representation, $\sigma \in \ZZ^{Q_0}$ a weight such that $\sigma \cdot \dd=0$. Let
$$
0=V_0\subseteq V_1 \subseteq \ldots \subseteq V_n=V
$$ 
be a filtration of subrepresentations of $V$ such that $\sigma \cdot \ddim V_i=0$ for $i \in \{0, \ldots, n \}$. Then there exists a $1$-parameter subgroup $\lambda: \RR^{\times} \to \GL(\dd)$ such that
$$
\lambda(t) \in \GL(\dd)_{\sigma}, \forall t \in \RR^{\times}, \text{~and~} \lim_{t \to 0} \lambda(t) V \simeq \bigoplus_{i=1}^n V_i/V_{i-1}.
$$ 
Consequently, $\bigoplus_{i=1}^n V_i/V_{i-1}$ belongs to the closure of the $\GL(\dd)_{\sigma}$-orbit of $V$.
\end{lemma}

In what follows, for a vector $\ff \in \RR^{Q_0}$, we denote by $\HH(\ff)$ the hyperplane in $\RR^{Q_0}$ orthogonal to $\ff$. We are now ready to prove one implication of Theorem \ref{quiver-RadIso-thm}.

\begin{lemma}\label{lemma-2-app} For a weight $\sigma \in \ZZ^{Q_0}$, if $V \in \rep(Q, \dd)$ is a $\sigma$-polystable representation then $\sigma \in \rel(\Omega(V))$.
\end{lemma}

\begin{proof} Let $\F$ be the face of $\Omega(V)$ that contains $\sigma$ in its relative interior. Let us write
\begin{equation}\label{eqn-1-proof-app}
\mathcal F=\Omega(V) \cap (\cap~ \HH(\ddim 
V')),
\end{equation}
where the intersection is over (finitely many) subrepresentations $V'$ of $V$. For each such subrepresentation $V' \subseteq V$, applying Lemma \ref{lemma-1-app} to the filtration $0\subseteq V' \subseteq V$, we get that $V' \oplus V/V' \in \overline{\GL(\dd)_{\sigma}V}$. Moreover, as $V$ is assumed to be $\sigma$-polystable, Proposition \ref{poly-stable-prop-app} tells us that $\GL(\dd)_{\sigma}V$ is closed in $\rep(Q, \dd)$, and so we obtain that $V'\oplus V/V' \simeq V$. This immediately implies that $\Omega(V) \subseteq \mathbb{H}(\ddim V')$ for every subrepresentation $V' \subseteq V$ that occurs in $(\ref{eqn-1-proof-app})$, and thus $\mathcal F=\Omega(V)$. This proves that $\sigma \in \rel(\Omega(V))$.
\end{proof}

In what follows, we denote by $W_{\CC}$ the complexification of a representation $W$ of $Q$. We recall that the orbit cone of a complex representation $X$ of $Q$ is the rational convex polyhedral cone consisting of all real weights $\theta \in \RR^{Q_0}$ such that $X$ is a $\theta$-semi-stable complex representation. Let $G:=\prod_{x \in Q_0}\GL(\dd(x), \CC)$ and $\mathbb{X}:=\prod_{a \in Q_1}\CC^{\dd(ha)\times \dd(ta)}$ on which $G$ acts by simultaneous conjugation. Note that $G_{\RR}$, the $\RR$-rational points of $G$, is $\GL(\dd)$, and $\mathbb{X}_{\RR}=\rep(Q, \dd)$. The following result, mostly proved in \cite[Proposition 2.4 and Remark 2.5]{HosSch2017}, allows us to transfer invariant-theoretic information between $\RR$ and $\CC$. This comes in handy since the invariant theory over $\CC$ is easier than that over $\RR$.

\begin{prop} \label{reals-complex-prop-app} Let $V \in \rep(Q,\dd)$ be a representation,
\begin{equation} \label{filtration-eqn-app}
0=V_0\subseteq V_1 \subseteq \ldots \subseteq V_n=V
\end{equation}
a filtration of $V$, and $\sigma \in \ZZ^{Q_0}$ a weight. Then the following statements hold.
\begin{enumerate}
\item For an integral weight $\theta \in \ZZ^{Q_0}$, $V$ is $\theta$-semi-stable if and only if $V_{\CC}$ is $\theta$-semi-stable. Consequently, 
$$\Omega(V)=\Omega(V_{\CC}).$$

\item If $V_{\CC}$ is $\sigma$-stable then $V$ is also $\sigma$-stable.

\item If $V$ is $\sigma$-stable then $V_{\CC}$ is either stable or a direct sum of two $\sigma$-stable (complex) representations. Consequently, if $V$ is $\sigma$-polystable then $V_{\CC}$ is also $\sigma$-polystable.
\item Assume that the factors $V_i/V_{i-1}$, $i \in \{1, \ldots, n\}$, are $\sigma$-stable and let $\gr_{\sigma}(V)=\bigoplus_{i=1}^n V_i/V_{i-1}$ be the associated graded module corresponding to the above filtration of $V$. Let $\gr_{\sigma}(V_{\CC}) \in \mathbb{X}$ be the associated graded module corresponding to a Jordan-H{\"o}lder filtration of $V_{\CC}$ in the category of $\sigma$-semi-stable complex representations of $Q$. Then
$$
\gr_{\sigma}(V)_{\CC}\simeq \gr_{\sigma}(V_{\CC}),
$$
and 
$$
\Omega(\gr_{\sigma}(V))=\Omega(V) \cap \left(\cap_{i=1}^n \HH(\ddim V_i) \right)
$$
is a face of $\Omega(V)$ containing $\sigma$.
\end{enumerate}
\end{prop}

\begin{proof} Parts $(1)-(3)$ are proved in  \cite[Proposition 2.4 and Remark 2.5]{HosSch2017}. For part $(4)$, complexifying the filtration $(\ref{filtration-eqn-app})$ we obtain a filtration of complex subrepresentations of $V_{\CC}$ whose factors $(V_{i})_{\CC}/(V_{i-1})_{\CC}=(V_i/V_{i-1})_{\CC}$, $i \in [n]$, are $\sigma$-polystable by part $(3)$. Thus $\gr_{\sigma}(V)_{\CC}$ is a $\sigma$-polystable representation that belongs to the closure of $G V_{\CC}$. Let us denote by $\mathbb{X}^{ss}_{\sigma}$ the $\sigma$-semi-stable locus in $\mathbb{X}$. Then it follows from King's work \cite{K} that the $G$-orbit of $\gr_{\sigma}(V)_{\CC}$ is the unique closed $G$-orbit in $\mathbb{X}^{ss}_{\sigma}$ lying in the closure of the $G$-orbit of $V_{\CC}$ in $\mathbb{X}^{ss}_{\sigma}$.  Since the same holds for the $G$-orbit of $\gr_{\sigma}(V_{\CC})$, we get that $\gr_{\sigma}(V)_{\CC}$ and $\gr_{\sigma}(V_{\CC})$ are isomorphic as complex representations of $Q$.

It remains to show that $\Omega(\gr_{\sigma}(V))$ is a face of $\Omega(V)$. Let $\theta \in \RR^{Q_0}$ be a real weight of $Q$. If $\theta \in \Omega(\gr_{\sigma}(V))$ then the direct summands $V_i/V_{i-1}$, $i \in [n]$, are also $\theta$-semi-stable which immediately implies that $\theta \in \HH(\ddim V_i)$ for all $i \in [n]$. Moreover, since the category of semi-stable representations is closed under extensions, we also get that $V_n=V$ is $\theta$-semi-stable. This proves the inclusion $\Omega(\gr_{\sigma}(V)) \subseteq \Omega(V) \cap \left(\cap_{i=1}^n \HH(\ddim V_i) \right)$. For the other inclusion, let $\theta \in \RR^{Q_0}$ be a real weight such that $V$ is $\theta$-semi-stable and $\theta \cdot \ddim V_i=0$ for all $i \in [n]$. Then each $V_i$ is $\theta$-semi-stable, implying that every quotient $V_i/V_{i-1}$ is $\theta$-semi-stable. Thus the direct sum $\bigoplus_{i=1}^n V_i/V_{i-1}$ is $\theta$-semi-stable and so $\theta$ belongs to $\Omega(\gr_{\sigma}(V))$. This now finishes the proof.
\end{proof}

\begin{remark} \label{rmk-iso-cplex-real} The category of representations of $Q$ can be identified with the category, $\module(A)$, of finite-dimensional left modules of the path algebra $A=\RR Q$ of $Q$ over $\RR$. If $M$ is a representation of $Q$ viewed as a left $A$-module then $M_{\CC}=M \otimes_{\RR} \CC$, the complexification of $M$, is a finite-dimensional module of $A_{\CC}=A \otimes_{\RR} \CC=\CC Q$, the path algebra of $Q$ over $\CC$. Moreover, $M_{\CC}=M \oplus iM\simeq M^{\oplus 2}$ as $A$-modules.

Now, let $M$ and $N$ be two finite-dimensional $A$-modules such that $M_{\CC} \simeq N_{\CC}$ as $A_{\CC}$-modules. In particular, they are isomoprhic as $A$-modules and thus
$$
M^{\oplus 2} \simeq N^{\oplus 2} \text{~as~} A\text{-modules}.
$$
Since $\module(A)$ has the Krull-Schmidt property, we get that $M \simeq N$ as $A$-modules, equivalently as representations of $Q$. For completeness, we mention that the converse of what we have just proved always hold, i.e. if $M$ and $N$ are isomorphic $A$-modules then $M_{\CC}$ and $N_{\CC}$ are clearly isomorphic as $A_{\CC}$-modules. 
\end{remark}

We are now ready to prove Theorem \ref{quiver-RadIso-thm}.

\begin{proof}[Proof of Theorem \ref{quiver-RadIso-thm}] We will show that 
$$\sigma \in \rel(\Omega(V)) \Longleftrightarrow V \text{~is~}\sigma\text{-polystable}.
$$  
This together with Proposition \ref{poly-stable-prop-app} will prove the claim of the theorem.

\bigskip
\noindent
$(\Longleftarrow)$ This implication is proved in Lemma \ref{lemma-2-app}.

\bigskip
\noindent
$(\Longrightarrow)$ We know from Proposition \ref{reals-complex-prop-app} that $\Omega(\gr_{\sigma}(V_{\CC}))$ is a face of $\Omega(V_{\CC})$ containing $\sigma$, a point in the relative interior of $\Omega(V_{\CC})$. Therefore,
$$
\Omega(\gr_{\sigma}(V_{\CC}))=\Omega(V_{\CC}).
$$

Since $V$ is assumed to be locally semi-simple, $V$ is $\sigma_0$-polystable for some weight $\sigma_0 \in \ZZ^{Q_0}$. By Proposition \ref{reals-complex-prop-app}{(3)}, this implies that $V_{\CC}$ is $\sigma_0$-polystable; in particular, $\sigma_0 \in \Omega(V_{\CC})=\Omega(\gr_{\sigma}(V_{\CC}))$, meaning that 
$$
\gr_{\sigma}(V_{\CC}) \in \mathbb{X}^{ss}_0,
$$ 
where $\mathbb{X}^{ss}_0$ is the $\sigma_0$-semi-stable locus in $\mathbb{X}$. Furthermore, according to King's work in \cite{K} (see also \cite[Appendix A]{DerMak-MLE-2020})), we know that $\V_{\CC}$ being $\sigma_0$-polystable is equivalent to the $G$-orbit of $V_{\CC}$ being closed in $\mathbb{X}^{ss}_0$, and so 
$$
\overline{GV_{\CC}} \cap \mathbb{X}^{ss}_0=GV_{\CC}.
$$

Next, since $\gr_{\sigma}(V_{\CC})$ is the associated graded module of a filtration of $V_{\CC}$, it can be viewed as the limit $\lim_{t \to 0} \lambda(t)V_{\CC}$ for a suitable $1$-parameter subgroup of $G$, and so $\gr_{\sigma}(V_{\CC}) \in \overline{G V_{\CC}}$. Putting everything together, we get that
$$
\gr_{\sigma}(V_{\CC}) \in \overline{GV_{\CC}} \cap \mathbb{X}^{ss}_0=GV_{\CC}.
$$
It now follows from Proposition \ref{reals-complex-prop-app}{(4)} that the complexifications of $\gr_{\sigma}(V)$ and $V$ are isomorphic. This further implies that $\gr_{\sigma}(V) \simeq V$ as real quiver representations by Remark \ref{rmk-iso-cplex-real}, and therefore $V$ must be $\sigma$-polystable.
\end{proof} 

\begin{rmk} 
In a sequel to the current work, we plan to address the algorithmic aspects of Theorem  \ref{RIF-thm-2} by using gradient descent to find and analyze algorithms that compute an approximation to a minimizing vector $t^{*}$ for our function $\Phi_{\F}(t)-\langle t, \cc \rangle, t \in \RR^n$. This combined with Theorem \ref{RIF-thm-2} will allow us to approximate the transformation $Q^{-{1 \over 2}}(t^*)$ to any desired level of accuracy. We point out that this task for the classical case (i.e. when $d_1=\ldots=d_n=1$) has been successfully carried out in \cite[Section 4]{Art-AviKapHai-2020}.  
\end{rmk}

\subsection*{Acknowledgment} The authors would like to thank Peter Casazza and Petros Valettas for many useful discussions on the paper. We are indebted to an anonymous referee for a very thorough report which helped improved the paper and for pointing out an error in an earlier version of inequality (\ref{maj-ineq}).

C. Chindris is supported by Simons Foundation grant $\# 711639$.


\end{document}